\documentclass[a4paper, 12pt]{article} 

\usepackage[left=3cm,right=2cm,top=2.5cm,bottom=2.5cm]{geometry}
\usepackage{lmodern}
 
\usepackage[T1]{fontenc}
\usepackage[utf8]{inputenc}

\usepackage{verbatim}

\usepackage{mathtools,amsmath,amsfonts,amssymb,amsthm,mathrsfs,stackrel}

\usepackage{tikz}
\usetikzlibrary{shapes,arrows,calc,positioning,patterns}

\usepackage{xargs,pgfplots}

\newcommandx{\change}[2][1=]{\todo[linecolor=blue,backgroundcolor=blue!25,bordercolor=blue,#1]{#2}}
\definecolor{p}{RGB}{2, 31, 146}
\definecolor{n}{RGB}{146, 2, 2}
\pgfplotsset{compat=1.17}

\usepackage[nottoc]{tocbibind}
\usepackage{hyperref}
\hypersetup{
    colorlinks=true,
    linkcolor=blue,
    filecolor=blue,
    urlcolor=blue,
    }
\usepackage{fancyhdr} 

\theoremstyle{definition}
\newtheorem{Rmk}{Remark}

\theoremstyle{plain}
\newtheorem{Theo}{Theorem}

\newtheorem{Lem}{Lemma}
\newtheorem{Cor}{Corollary}
\newtheorem{Prop}{Proposition}
\newtheorem*{Theo*}{Theorem}
\newtheorem*{Lem*}{Lemma}
\newtheorem*{Cor*}{Corollary}
\newtheorem*{Prop*}{Proposition}

\newtheorem{innercustomgeneric}{\customgenericname}
\providecommand{\customgenericname}{}
\newcommand{\newcustomtheorem}[2]{%
  \newenvironment{#1}[1]
  {%
   \renewcommand\customgenericname{#2}%
   \renewcommand\theinnercustomgeneric{##1}%
   \innercustomgeneric
  }
  {\endinnercustomgeneric}
}
\newcustomtheorem{customthm}{Theorem}
\newcustomtheorem{customlemma}{Lemma}

\newcommand{\ind}{\mathbf{1}}
\newcommand{\Var}[1]{\mathrm{Var}\left(#1\right)}

\newcommand{\Ex}[1]{\mathbb{E}\left(#1\right)}
\renewcommand{\P}[1]{\mathbb{P}\left(#1\right)}
\newcommand{\Pa}[1]{\mathbb{P}_{\alpha}\left(#1\right)}
\newcommand{\Z}{{\mathbb{Z}}}
\newcommand{\N}{{\mathbb{N}}}
\newcommand{\R}{\mathbb{R}}
\newcommand{\T}{\mathbb{T}}
\newcommand{\G}{\mathrm{Gumbel}}
\newcommand{\e}{\varepsilon}

\begin{document}

\title{Covering Distributions}
\author{Alberto M. Campos\footnote{IMPA - Estrada Dona Castorina 110,
22460-320 Rio de Janeiro, RJ - Brazil }\; and  Augusto Teixeira \footnote{IMPA - Estrada Dona Castorina 110,
22460-320 Rio de Janeiro, RJ - Brazil and IST - University of Lisbon,
Portugal.}}
\date{\today}
\maketitle


\begin{abstract}
In this article, we study a covering process of the discrete one-dimensional torus that uses connected arcs of random sizes in the covering. More precisely, fix a distribution $\mu$ on $\mathbb{N}$, and for every $n\geq 1$ we will cover the torus $\mathbb{Z}/n\mathbb{Z}$ as follows: at each time step, we place an arc with a length distributed as $\mu$ and a uniform starting point. Eventually, the space will be covered entirely by these arcs. Changing the arc length distribution $\mu$ can potentially change the limiting behavior of the covering time. Here, we expose four distinct phases for the fluctuations of the cover time in the limit. These phases can be informally described as the Gumbel phase, the compactly support phase, the pre-exponential phase, and the exponential phase.  Furthermore, we expose a continuous-time cover process that works as a limit distribution within the compactly support phase. 
\end{abstract}

{\footnotesize Keywords: Coupon Collector ; Covering Processes;}
\section{Introduction and motivation}
\label{sec:Introduction}\noindent

Consider a scenario in which objects of random size fall randomly into some space and, over time, as more objects enter the picture, the space will eventually be covered. This general description shares similarities with several other problems in the literature, being a generalization of the coupon collector problem \cite{ER,BP,DB}, the committee problem \cite{MS}, the leaves-covering problem \cite{MDP}, the set-covering problems \cite{MFBBBRWK, Sh, ADiaconis}, or the matrix occupancy problem \cite{LH}. Each of these works cited above gives a different perspective to the covering process. However, there are still many open problems to be understood.    

This paper focuses on a specific type of covering process that allows us to observe different phases for the fluctuations of the cover time. Fix the space $ \mathbb{Z}/n\mathbb{Z}$ to be a discrete one-dimensional torus of size $n$ represented by the points $\{0,...,n-1\}$, and consider the set of connected arcs in the torus. Informally speaking, the process can be described as a random stack of arcs, where the arcs are chosen to start at random points $U$ and have random lengths $R$. 

To describe it rigorously, pick two sequences of i.i.d. random variables. The first sequence is going to represent the \textbf{radii} $(R_k)_k$, where
\begin{align}\label{eq:Rdistribuição}
    \P{R_1\geq 1}=1 \text{ and }
    \P{R_1\geq r}\coloneqq&\,f(r) \text{ for } r\geq 1.
\end{align} The second sequence of random variables will be the \textbf{positions} $(U_k)_k$, where $U_k$ is distributed as a uniform random variable in the discrete torus, i.e. $U_i\sim \mathrm{Unif}(\mathbb{Z}/n\mathbb{Z})$. With these two sequences, define the set $\{\Gamma_m=[0,m) :m\in \N\}$ and denote the $k-th$ object as $\mathcal{O}_k=\{U_k+\Gamma_{R_k}\}=\{U_k,U_k+1,...,U_k+R_k-1\}\subset \mathbb{Z}/n\mathbb{Z}$, a interval that starts in $U_k$ and has total length $R_k$. Finally, define the \textbf{discrete covering process} as the sequence $({C}_k)_k$, where ${C}_0=\emptyset$ and inductively ${C}_k={C}_{k-1}\cup \mathcal{O}_k$.

With the covering process $({C}_k)_k$ defined, consider the \textbf{cover time} of the space $\mathbb{Z}/n\mathbb{Z}$ 
\begin{align*}
    \tau_n=\min\{ k : {C}_k=\mathbb{Z}/n\mathbb{Z}\}.
\end{align*} 
The main focus of the current article is to investigate how the distribution of the radius $R$ can influence the scaling and limiting distribution of $\tau_n$. 

The process translates easily when working in continuous time, resulting in simpler proofs and calculations. In order to define it, let $\{N(t):\, t\geq 0\}$ be a Poisson process with rate $1$ on the real line and fix a discrete covering process $({C}_k)_k$ with radius distribution $f(r)=\P{R\geq r}$. Then, define the \textbf{continuous time covering process} $(X_t)_{t>0}$ as
\begin{align}\label{eq:ContdefX}
    X_t={C}_{N(t)}.
\end{align}Define also the \textbf{continuous cover time} as
\begin{align*}
    T_n=\inf\{t: X_t=\mathbb{Z}/n\mathbb{Z}\}.
\end{align*} \bigskip

Theorems \ref{teo:1}, \ref{teo:2}, \ref{teo:3}, and \ref{teo:4}, stated below, describe different phases of the covering process. To illustrate thes, see Figure~\ref{fig:1}: 

\begin{figure}[ht!]
    \centering
    \includegraphics{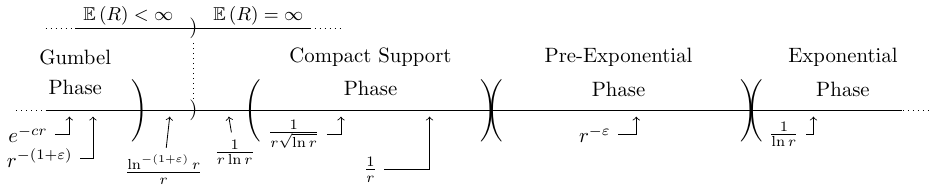}
    \caption{A representation of the Theorems \ref{teo:1} to \ref{teo:4} in a line, with the radius distributions $f(r)=\P{R\geq r}$ disposed in a monotone way. 
    Here $\e>0$ is a positive small number, and $c>0$ is any constant.}
    \label{fig:1}
\end{figure}

The covering process described above is a flexible model, where similar results using different notations appear scattered in the literature. But as far as we can tell, the theorems presented here are new.

Each theorem stated here presents a novelty in relation to the existing literature. To contextualizeand structure the statements, the references are presented between each. 

\begin{customthm}{A}[Gumbel Phase]\label{teo:1}
  Assume that $\Ex{R^{1+\e}}<\infty$ for some $\e>0$, and set $\mu=\Ex{R}$. Then, as $n$ goes to infinity:
\begin{align} \label{eq:DiscreteGumbelResult}
    \frac{\mu}{n}T_n - \ln{n}\overset{D}{\implies}\G(0,1),
\end{align} where $\mathbb{P}\left(\G(0,1)<t\right)=\exp\{-\exp\{-t\}\}$  is the Gumbel distribution with parameters $0$ and $1$.   
\end{customthm}

 Theorem \ref{teo:1} deals with the problem where the radius $R$ has light tails. Due to the number of available techniques, it is not surprising that this type of covering is the most explored in the literature. For example \cite{FLM} establishes the first-order behavior of the cover time, while \cite{FLM,BT,XT} show convergence in distribution when $R$ is constant. 
 \begin{Rmk}\label{Rmk:SJ}
 The article that most resembles in results with Theorem \ref{teo:1} is \cite{SJ}, where the author considers $\widehat{\tau_n}$, the number of objects needed to cover a continuous circle $\mathbb{S}^1$ using random arcs that start at uniform points and have lengths $R/n$, where $\Ex{R^{1+\e}}<\infty$ for some $\e>0$. This covering problem has the following scaling behavior: 
\begin{align*}
    \frac{\mu}{n}\widehat{\tau_n} -\ln{\left(\frac{n}{\mu}\right)}-\ln\left(\ln{\left(\frac{n}{\mu}\right)}\right)\overset{D}{\implies} \G(0,1). 
\end{align*} The above result has a different normalization scale then Theorem \ref{teo:1}. Intuitively speaking, the cover time in the continuum is larger because covering the discrete points of the form $\left\{\frac{i}{n} : i\in \{0,1,\cdots,n-1\}\right\}$ does not imply that points in between are covered too. Also, it is important to mention that the techniques presented in \cite{SJ} are different from those discussed in here. 
\end{Rmk}\bigskip

A new phenomenon emerges when the radius stops having a first moment; not only does the cover time grow on a different scale, but its limiting distribution displays a compact support.

\begin{customthm}{B}[Compact Support Phase]\label{teo:2}
  Let $b>-1$, and assume that $f(r)= \min\{\frac{\ln^{b}(r)}{r},1\}$. Then, $(f(n)\tau_n)_n$ is tight. Moreover, for every subsequence $(n_k)_k$ such that
  \begin{align} \label{eq:DiscretecompactResult}
    f(n_{k})T_{n_{k}}\overset{D}{\implies} Y,
  \end{align}  we conclude that $Y=Y(f,(n_k)_k)$ is a non-degenerate distribution with compact support.
\end{customthm}

We can generalize the result of Theorem \ref{teo:2} for other choices of functions $f(r)$. Due to some technicalities, we postponed the general statement for the compact support phase to Section \ref{sec:compactphase}.

The particular case $(b=0)$ features some self-similarity that reassembles a previously studied covering process. To define it informally, consider $\mathbb{S}^1$ a circle with a canonical representation $[0,1)$ where $1$ is identified to $0$; and let $(\mathcal{C}_t)_{t>0}$ be a set of arcs on $\mathbb{S}^1$, with initial points $x$ and radii $y$ coming from a Poisson point process with rate $d\Lambda=\alpha \left(dx\otimes \frac{dy}{y^2}\right)$; for more details, see a complete description of the model in Section~\ref{sec:cylinder}. Defined by Mandelbrot and Shepp in \cite{BBM,Sh}, the covering process $(\mathcal{C}_{t})_{t>0}$ has been studied in \cite{KA,JB}. The novelty of Theorem \ref{teo:2*} is related to the fact that the discrete process converges (in a sense) to its continuous counterpart $(\mathcal{C}_t)_{t>0}$.


\begin{customthm}{B*}
  \label{teo:2*}
  Letting $f(r)=\frac{1}{r}$, then:
  \begin{align} \label{eq:limitdiscrete}
    \frac{T_{n}}{n}\overset{D}{\implies} Y,
  \end{align} where $Y$ is a non trivial distribution with $\mathrm{supp}\{Y\}=[0,1]$, that satisfies for $\alpha\in(0,1)$:
  \begin{align*}
    \P{\mathcal{C}_{\alpha}=\mathbb{S}^1}&=\P{Y\leq \alpha}, \text{ and }\\
    \frac{\log |\T_n\setminus X_{\alpha n}|}{\log n}\ind\{X_{\alpha n}\neq \T_n\}&\overset{D}{\implies} (1-\alpha)\ind\{\mathcal{C}_{\alpha}\neq \mathbb{S}^1\}.
  \end{align*} 
\end{customthm}

In this paper, we define the \textbf{support of a distribution} $Y$ to be the set of all points $y\in \R$ such that for every neighborhood $U$ of $y$, we get $\P{Y\in U}>0$. 

Given any finite number of points in $\mathbb{S}^1$, the probability of covering these points in the discrete case converges to one. However, the cover time is not a random variable that depends only on a finite set of points, so there is a tightness issue that is central to Theorem \ref{teo:2*}. Showing that the discrete cover time converges to its continuous analogue is not trivial, as we need to control the random variables beyond just a local perspective. Remark \ref{Rmk:SJ} serves as an cautionary tail that the discrete and the continuous processes may feature different behaviors.
\bigskip

To state the next result, following the notation in \cite{SIR}, we say that
 a measurable function $U:\R_+\to \R_+$ is regularly varying at infinity with index $p\in\R$ if for all $t>0$, we have: 
\begin{align}\label{eq:RV}
    \lim_{x\to \infty}\frac{U(xt)}{U(x)}=t^{p}.
\end{align} Let $\mathrm{RV}_p$ be the space of all \textbf{regularly varying functions at infinity with index $p$}, and when equation \eqref{eq:RV} holds, denote $U\in \mathrm{RV}_p$. In particular, when $p=0$, denote $\mathrm{RV}_0$ as the space of \textbf{slowly varying functions}.

When exploring heavier tails for $R$, limiting distributions with unbounded support are once again found.

\begin{customthm}{C}[Pre-Exponential Phase]\label{teo:3}
  Take $p\in(-1,0)$, and set $f\in \mathrm{RV}_p$. Then, $(f(n)T_n)_n$ is tight. Moreover, for every subsequence $(n_k)_k$ such that
 \begin{align} \label{eq:DiscretePExpResult}
    f(n_{k})T_{n_{k}}\overset{D}{\implies} Z,
\end{align} we can conclude that  $Z=Z(f,(n_k)_k)$ is a distribution that satisfies 
 \begin{align}\label{eq:DiscretePExpZcondition}
   1 - e^{-z} < \mathbb{P}(Z \geq z) < 1, \text{ for every } z\geq 0.
 \end{align} In particular, it is not compactly supported nor exponentially distributed. 
\end{customthm}

\begin{Rmk}
     With $p\in(-1,0)$, as a standard example of a function in $\mathrm{RV}_p$, one may take $f(r)=r^{p}$.
\end{Rmk}

Note that equation \eqref{eq:DiscretePExpZcondition} rules out the possibility that $Z$ is an exponential random variable. This observation highlights the contrast with the next theorem, which focuses on the heaviest tail considered in this paper and exhibits an exponential distribution in its limit. In this case, the system waits for the arrival of a single interval that covers the whole torus. 

\begin{customthm}{D}[Exponential Phase]\label{teo:4}
  Let $f$ be a slowly varying function, then
  \begin{align} \label{eq:DiscreteExpResult}
    f(n)T_n\overset{D}{\implies} \mathrm{Exp}\{1\},
  \end{align}where $\mathbb{P}\left(\mathrm{Exp}\{1\}<t\right)=1-e^{-t}$ is the exponential distribution with parameter $1$.  
\end{customthm}
\begin{Rmk}
As examples of functions in $\mathrm{RV}_0$, consider functions such as $\frac{1}{\log{n}}$ or $\frac{1}{\log^b{n}}$ for some $b > 0$.
\end{Rmk}
\begin{Rmk}
    The results appearing in Theorems \ref{teo:1}, \ref{teo:2}, \ref{teo:3}, and \ref{teo:4}, are also valid in the discrete covering process changing only the random variable $T_n$ to $\tau_n$, see Appendix Subsection \ref{subsec:Equivalence}. 
\end{Rmk}

In this text, the proof of theorems \ref{teo:1}, \ref{teo:2}, \ref{teo:3}, and \ref{teo:4}, are presented, respectively, in sections~\ref{sec:Gumbel}, \ref{sec:compactphase}, \ref{sec:preexpphase}, and \ref{sec:expphase}.  Section~\ref{sec:cylinder} is dedicated to proving theorem \ref{teo:2*} and construct the continuous model that serves as a limit distribution.  The proofs are independent and can be read in any order.  Section~\ref{subsec:Open} is dedicated to open problems that remains unsolved.  Finally, Appendix~\ref{sec:appendix} is divided into two subsections, Subsection~\ref{subsec:Usefulprop} presents some lemmas and useful tools used in the proof of the theorems, and Subsection~\ref{subsec:Equivalence} demonstrates the theorems in the discrete time case.

\textbf{Notations:} In the text, the symbol $\N$ is used to represent the set of natural numbers $\{1, 2, ...\}$. The set $\Z$ is used to denote the integer numbers, while $\R$ is used for the real numbers. The circle is defined as $\mathbb{S}^1=\R/\Z$. Furthermore, the sets $I$ and $J$ are utilized as arbitrary index sets.

\textbf{Acknowledgment:} 
Both authors thanks  CNPQ (Conselho Nacional de Desenvolvimento Científico Tecnolôgico), IST (Instituto Superior Técnico) and IMPA (Instituto de Matemática Pura e Aplicada). Researcher  A.M.C. is supported in part by CNPq, grant 141068/2020-5.
Researcher A.T. has been supported by grants \textit{``Projeto Universal''} (406250/2016-2)
and \textit{``Produtividade em Pesquisa''} (304437/2018-2) from CNPq and \textit{``Jovem
Cientista do Nosso Estado''}, (202.716/2018) from FAPERJ. 
 

\section{Gumbel Phase}\label{sec:Gumbel}
\noindent

This Section is devoted to the proof of Theorem \ref{teo:1}. We organized it as follows: First, we expose and prove some lemmas about the behavior of the light tail radius. Then, we expose three propositions that together prove Theorem \ref{teo:1}. Finally, we present the proof of the propositions.

\begin{Lem} \label{Lem:EqvH}
Let $R$ be a discrete random variable with $\mathbb{P}(R \geq r) = f(r)$. The following statements are equivalent:
\begin{align}
    \textit{1.}&\,\quad \Ex{R^{p+1}}<\infty \text{ for some } p>1. \nonumber\\
    \textit{2.}&\,\quad \exists\lambda>0 \text{ such that } \lim_{k\to \infty} f(k)k^{1+\lambda}=0. \label{eq:ContinuousGumbelCondition2}\\
    \textit{3.}&\,\quad  \exists \lambda'>0\text{ such that }\lim_{k\to \infty} f(k)k^{1+\lambda'}\ln{k}=0. \label{eq:ContinuousGumbelCondition}
\end{align}
\end{Lem}

Lemma \ref{Lem:EqvH} plays a role in simplifying several proofs in this Section. The proof of it is straightforward and can be found in the Appendix Subsection \ref{subsec:Usefulprop}. Continuing with the lemmas, let $\mu=\Ex{R}$ and define the auxiliary variable:
\begin{align}\label{eq:Gumbel-2}
    g_k=\frac{\sum_{i=1}^{k}f(i)}{\mu}.
\end{align} About this auxiliary sequence,  it is trivial to check that $g_1=\mu^{-1}$ and $g_k$ is a non-decreasing sequence that converges to one. Another important observation is given below.

\begin{Lem} \label{Lem:Gumbel-1}
If $f$ satisfies condition \eqref{eq:ContinuousGumbelCondition2}, i.e. exists $\lambda>0$ such that $\lim_{k\to \infty} f(k)k^{1+\lambda}=0$. Then
\begin{align}
     \lim_{n\to \infty} (1-g_n)\ln{n}&=0.  \label{eq:Gumbel-4}
\end{align}
\end{Lem}
\begin{proof}[Proof of Lemma \ref{Lem:Gumbel-1}]
Using condition \eqref{eq:ContinuousGumbelCondition2}, for some $\lambda>0$, exists a $n_0(\lambda)$ and a constant $C=C(\lambda)$ such that for every $n>n_0$:
\begin{align*}
    (1-g_n)\ln{n}&=\frac{\ln{n}}{\mu}\sum_{i=n+1}^{\infty} f(i)\leq \frac{C\ln{n}}{n^{\lambda}\mu}.
\end{align*} The proof follows by taking the limit.
\end{proof}

Now, fix an arbitrary $\alpha\in(0,1)$ and define the $\alpha-vacant$ set
\begin{align*}
    \mathcal{V}_{\alpha}=\left(\mathbb{Z}/n\mathbb{Z}\right)\setminus X_{\alpha \frac{n\ln{n}}{\mu}},
\end{align*} this is the set of points in $\mathbb{Z}/n\mathbb{Z}$ that have not yet been covered at time $\alpha \frac{n\ln{n}}{\mu}$.


The first step of the proof consists in showing that $\mathcal{V}_{\alpha}$ is a set of sparse points that has a polynomial size; see Propositions \ref{Prop1} and \ref{Prop2}. The second step relates the cover time of the sparse set $\mathcal{V}_{\alpha}$ with an time scale of a classical coupon collector. In this relation, we prove that the cover happens one point at a time; see Proposition \ref{Prop3}. With that in hand, the proof of Theorem \ref{teo:1} follows by simple observations made in Subsection~\ref{subsec:Gumbel-ProofofTheorem}.

For any $\alpha,\beta\in (0,1)$, define the event
\begin{align*}
    A_\alpha(\beta)=\left\{\exists x,y\in \mathcal{V}_{\alpha} \,: \, |x-y|<n^{\beta}\right\}.
\end{align*}

\begin{Prop}[Sparse]\label{Prop1}
  Choose $R$ satisfying $\Ex{R}=\mu<\infty.$ Then, for every choice of $\beta\in(0,1)$ there exists $\alpha=\alpha(\beta,\mu)<1$ such that
\begin{align}\label{eqProp1}
    \lim_{n\to \infty}\P{A_{\alpha}(\beta)}= 0.
\end{align}
\end{Prop}

The proof of this proposition exposed in Subsection ~\ref{subsubsection:ProofofProp1} is based on a union bound argument on the set $A_{\alpha}(\beta)$. 

\begin{Prop}[Concentration]\label{Prop2}
For all $\alpha<1$, and for any given $\delta>0$
 \begin{align} \label{eqProp2}
     \lim_{n\to \infty} \P{\left||\mathcal{V}_{\alpha}|-n^{1-g_n\alpha}\right|\geq \delta n^{1-g_n\alpha}}= 0.
 \end{align}
\end{Prop}

The proof of this proposition exposed in Subsection~\ref{subsubsection:ProofofProp2} relies on a second moment estimate of the random variable $|\mathcal{V}_{\alpha}|$. 

To state Proposition \ref{Prop3}, let $\beta\in(0,1)$ and define the family of sets
\begin{align*}
    \mathcal{K}_{\beta}^n=\{K\subset \mathbb{Z}/n\mathbb{Z} : \forall x,y\in K, |x-y|>n^{\beta}\}.
\end{align*} Also, for any set $K\in \mathcal{K}_{\beta}$ define the time to cover $K$ as $T_K=\inf\{t: K\subset X_t\}$.

\begin{Prop}[Covering a sparse set]\label{Prop3}
Take $R$ satisfying the condition \eqref{eq:ContinuousGumbelCondition}, i.e., there exists $\lambda>0$ such that $\lim_{k\to \infty} f(k)k^{1+\lambda}\ln{k}=0$. With fixed $R$, exists $\beta_0=\beta_0(\lambda)<1$, such that for every $\beta\in (\beta_0,1)$ and any sequence of sets $\{K(n)\}_n$ that satisfies $\lim_{n}|K(n)|=\infty$ and $K(n)\in \mathcal{K}_{\beta}^n$ for every $n>0$, we get that: 
\begin{align}\label{eqProp3}
    \frac{\mu}{n}T_{K(n)}-\ln{|K(n)|}\overset{D}{\implies} \G(0,1),
\end{align}
\end{Prop}

The proof of Proposition \ref{Prop3}, presented in Subsection \ref{subsubsection:ProofofProp3}, obtains \eqref{eqProp3} by creating a coupling between the covering of the set $K(n)$ and a time change of the classical Coupon collector problem, since it will be proved that with high probability the covering of $K(n)$ happens one point at a time.

\subsection{Proof of Theorem \ref{teo:1}} \label{subsec:Gumbel-ProofofTheorem}

With no further delay, assuming all the tree propositions, Theorem \ref{teo:1} will be proven in this Subsection. The subsequent Subsection \ref{subsec:Gumbel-ProofofPropositions} will contain the proofs of the propositions. 

\begin{proof}[\textbf{Proof of Theorem \ref{teo:1}}] The idea of the proof consist in using the hypotheses and Propositions \ref{Prop1} and \ref{Prop2} to find a sparse and large set of vacant points. Next, using Proposition \ref{Prop3}, when covering this sparse and large set, the Gumbel distribution will appear. 

Taking $R$ such that $\Ex{R^p}<\infty$ for some $p>1$, and $\Ex{R}=\mu$, by Lemma \ref{Lem:EqvH} one can find $\lambda>0$ such that $f(k)k^{1+\lambda}\ln{k}$ goes to zero, when $k$ goes to infinity. For this fixed $\lambda>0$,  find $\beta_0(\lambda)$ using Proposition \ref{Prop3}, and fix any $\beta\in (\beta_0,1)$ to control how sparse the set needs to be. Now, with $\beta$ fixed, use Proposition \ref{Prop1} to fix some $\alpha\in(\alpha(\beta,\mu),1)$.

With the parameters fixed, start using Proposition \ref{Prop1} and \ref{Prop2} to show that the set
\begin{align*}
    \Omega_{n}=\left\{\omega: |\mathcal{V}_{\alpha}(\omega)|>\ln{n}, \mathcal{V}_{\alpha}(\omega)\in \mathcal{K}^n_{\beta}\right\}
\end{align*} has probability converging to one. To exactly compute how big in size the vacant set is, use the full strength of Proposition \ref{Prop2} to get
\begin{align}\label{eq:Gumbel-5}
    \ln\left(\frac{n^{1-g_n \alpha}}{|\mathcal{V}_{\alpha}|}\right)\overset{P}{\longrightarrow}0,
\end{align} where the right arrow $P$ indicates convergence in probability.

\quad On the event $\Omega_n$, define $T_{\alpha}=\inf\{t>0: \mathcal{V}_{\alpha}\subset X_{t+\alpha \frac{n\ln{n}}{\mu}}\}$ the cover time of the vacant set $\mathcal{V}_{\alpha}$. In particular, one can observe that
\begin{align}\label{eq:Gumbel-6}
    \ind_{\Omega_n}T_{\alpha}=\ind_{\Omega_n}\left(T_n-\alpha \frac{n\ln{n}}{\mu}\right),
\end{align} where $T_n=\inf\{t: \mathbb{Z}/n\mathbb{Z}\subset X_t\}$. Since it is also the case that $|\mathcal{V}_{\alpha}|$ diverges on the event $\Omega_n$, apply Proposition~\ref{Prop3} to obtain
\begin{align}\label{eq:Gumbel-7}
    \ind_{\Omega_n}\left(\frac{\mu}{n} T_{\alpha}-\ln{|\mathcal{V}_{\alpha}|}\right) \overset{D}{\implies} \G(0,1),
\end{align} when $n\to \infty$;  Therefore, using equations \eqref{eq:Gumbel-7} and \eqref{eq:Gumbel-5}:
\begin{align*}
     \ind_{\Omega_n}\left(\frac{\mu}{n}T_{\alpha} -(1-\alpha g_n)\ln{n}\right)\overset{D}{\implies}\G(0,1).
\end{align*} Now, by definition \eqref{eq:Gumbel-6}, we get that:
\begin{align*}
    \ind_{\Omega_n}&\left(\frac{\mu}{n}\left(T_n-\alpha \frac{n\ln{n}}{\mu}\right) -(1-\alpha g_n)\ln{n}\right)\nonumber\\&=\ind_{\Omega_n}\left(\frac{\mu}{n}T_{n} -(1+\alpha (1-g_n))\ln{n}\right)\overset{D}{\implies}\G(0,1).
\end{align*}Given any sequence of random variables $(Y_n)_n$, and $\Omega_n$ that satisfies: $Y_n \ind_{\Omega_n}\overset{d}{\implies} Y$ and $\P{\Omega_n}\to 1$ when $n\to \infty$. Then $Y_n \overset{d}{\implies} Y$. In our case, using condition \eqref{eq:Gumbel-4} and  the fact that $\P{\Omega_n}\to 1$ when $n\to \infty$, the equation lead to:
\begin{align*}
    \frac{\mu}{n}T_n -\ln{n}\overset{D}{\implies}\G(0,1),
\end{align*}proving the Theorem \ref{teo:1}. 
\end{proof}

\subsection{Proof of Propositions \ref{Prop1}, \ref{Prop2} and \ref{Prop3}}
\label{subsec:Gumbel-ProofofPropositions}

This Subsection aims to prove all three propositions. Before doing this, let us start by investigating the probability that a single site remains vacant. 

\begin{Lem} \label{Lem:Probability2}
On the continuous cover process at time $\alpha \frac{n\ln{n}}{\mu}$
\begin{align}\label{eq:Proofpropequation1}
    \P{0\in \mathcal{V}_{\alpha}}&=n^{-\alpha g_n}, \text{ and }\\
    \P{0,k\in \mathcal{V}_{\alpha}}&=n^{-\alpha(g_k+ g_{n-k})}\label{eq:Proofpropequation2}
\end{align} for any $k\in\{1,...,n-1\}$.
\end{Lem}
\begin{proof}
Observe that the number of objects covering $0$ corresponds to a Poisson random variable. The computation of its rate involves determining the probability that a single object covers the origin. To proceed, we have:
\begin{align*}
    \P{0\in \mathcal{O}_1}&=\sum_{i=1}^{n}\P{U_1=i-1, R\geq i}=\sum_{i=1}^{n}\frac{f(i)}{n}.
\end{align*} In particular, the rate of the Poisson is given by the product
\begin{align*}
    \alpha\frac{n\log(n)}{\mu}\sum_{i=1}^{n}\frac{f(i)}{n}= \alpha g_n \ln(n). 
\end{align*} Therefore:
\begin{align*}
    \P{0\in \mathcal{V}_{\alpha}}&=\exp\{-\alpha g_n\ln{n}\},
\end{align*} proving \eqref{eq:Proofpropequation1}.

The proof of the statement of equation \eqref{eq:Proofpropequation2} relies on counting the number of objects that at time $\alpha \frac{ n\ln{n}}{\mu}$ hits $0$ or $k$ for some $k\in\{1,...,n-1\}$. For this, observe that: When the uniform $U$ is between $1$ and $k$, the object need just to hit $k$; When it is between $k+1$ and $n$, the object need to hit the origin. So:
\begin{align*}
\mathbb{P}\left(\{0\in \mathcal{O}_1\}\cup \{k\in \mathcal{O}_1\}\right)=\frac{\sum_{i=1}^{k}f(i)}{n}+\frac{\sum_{i=1}^{n-k}f(i)}{n}. 
\end{align*} Therefore:
\begin{align*}
    \P{0,k\in \mathcal{V}_{\alpha}}&=\exp\{-\alpha (g_k+g_{n-k})\ln{n}\},
\end{align*} proving \eqref{eq:Proofpropequation2}.
\end{proof}

\subsubsection{Proof of Proposition \ref{Prop1}}
 \label{subsubsection:ProofofProp1}

\begin{proof}[Proof of Proposition \ref{Prop1}]
Let $\mu=\Ex{R}$, then for any choice of $\beta\in (0,1)$, take $\alpha=\alpha(\mu,\beta)$ satisfying:
\begin{align*}
    1> \alpha >\max\left\{\frac{1}{1+\mu^{-1}},\, \frac{1+\beta}{2}\right\}.
\end{align*} and choose $\e=\e(\mu,\alpha,\beta)$ such that: 
\begin{align}\label{eq:echoose}
    0<\e<\min\left\{2-\frac{(1+\beta)}{\alpha},\frac{2(\alpha(1+\mu^{-1})-1)}{\alpha}\right\}
\end{align} Using a union bound on $A_{\alpha}(\beta)$, and relation \eqref{eq:Proofpropequation2}, the computation leads to:
\begin{align*}
    \P{A_{\alpha}(\beta)}&=\mathbb{P}\left(\bigcup_{x,y\in \T_{n}:\, |x-y|<n^{\beta}} \{x,y\in \mathcal{V}_{\alpha} \}\right)\\
    &\leq n \sum_{k=1}^{\lfloor n^{\beta}\rfloor } \P{0,k \in \mathcal{V}_{\alpha}}=\sum_{k=1}^{\lfloor n^{\beta}\rfloor} n^{1-\alpha(g_k+g_{n-k})}.
\end{align*}

With fixed $\beta$, and with $\alpha$ and $\e$ chosen accordingly. Since $g_1=\mu^{-1}$, and $(g_n)_n$ converges monotonously to one, find $k_0=k_0(\e)$ such that for every $k>k_0$ we have $g_k>1-\e/2$. Then, for $n>2k_0$:
\begin{align*}
    \P{A_{\alpha}(\beta)}&\leq \sum_{k=1}^{k_0} n^{1-\alpha(g_k+g_{n-k})} + \sum_{k=k_0+1}^{\lfloor n^{\beta}\rfloor }n^{1-\alpha(g_k+g_{n-k})}\\
    &\leq \sum_{k=1}^{k_0} n^{1-\alpha(g_k+1-\e/2)} +\sum_{k=1}^{\lfloor n^{\beta}\rfloor} n^{1-\alpha(2-\e)}\\
    &\leq k_0n^{1-\alpha(1+\mu^{-1}-\e/2)}+n^{1+\beta-\alpha(2-\e)}
\end{align*} By the choice of  $\e$ and $\alpha$, it follows that $\P{A_{\alpha}(\beta)}$ decays polynomialy to zero.
\end{proof}

\subsubsection{Proof of Proposition \ref{Prop2}}
\label{subsubsection:ProofofProp2}

\begin{proof}[Proof of Proposition \ref{Prop2}]
The proof follows by Chebyshev's inequality. Start, by using Lemma \ref{Lem:Probability2} to get:
\begin{align*}
    \Ex{|\mathcal{V}_{\alpha}|}&=n^{1-\alpha g_n}, \text{ and }\\
    \Ex{|\mathcal{V}_{\alpha}|^2}&=\sum_{x\in \mathbb{Z}/n\mathbb{Z}}\P{x\in \mathcal{V}_{\alpha}}+\sum_{\substack{x,y\in \mathbb{Z}/n\mathbb{Z}\\ x\neq y}}\P{x,y\in \mathcal{V}_{\alpha}}\\
    &= n^{1-\alpha g_n}+2n\sum_{k=1}^{n/2}n^{-\alpha(g_k+g_{n-k})}.
\end{align*} Now, applying the Chebyshev's inequality  it follows for any $\delta>0$ that
\begin{align}
    \P{\left||\mathcal{V}_{\alpha}|-\Ex{|\mathcal{V}_{\alpha}|}\right|>\delta n^{1-\alpha g_n}}&\leq \frac{\Ex{|\mathcal{V}_{\alpha}|^2}-\Ex{|\mathcal{V}_{\alpha}|}^2}{n^{2(1-\alpha g_n)}\delta^2} \nonumber\\
    &=\frac{n^{1-\alpha g_n}-n^{2(1-\alpha g_n)}+2n\sum_{k=1}^{n/2}n^{-\alpha(g_k+g_{n-k})}}{n^{2(1-\alpha g_n)}\delta^2} \nonumber\\
    \label{eq:Gumbel-8} &=\frac{1}{\delta^2n^{1-\alpha g_n}}+\frac{1}{\delta^2}f(\alpha,n),
\end{align} where:
\begin{align*}
    f(\alpha,n)=\frac{2}{n}\sum_{k=1}^{n/2}(-1+n^{\alpha(2g_n-g_k-g_{n-k})}).
\end{align*} Since $\alpha<1$, and the distribution $f$ satisfies condition \eqref{eq:Gumbel-4} the proofs follows directly of the following Lemma.
\end{proof}
\begin{Lem}
If condition \eqref{eq:Gumbel-4} is satisfied, that is $\lim_n (1-g_n)\ln{n}=0$, and $\alpha<1$, then $\lim_{n\to \infty}  f(\alpha,n)=0.$
\end{Lem}
\begin{proof}
Fix $\alpha<1$, $\mu>1$ and take $\gamma\in\left(0,\frac{1}{2}\right)$, such that $\gamma<2(1-\alpha(1-\mu^{-1}))$. Since $(g_n)_n$ converges monotonously to one, choose $k_0=k_0(\alpha,\gamma)$ such that $1-g_k\leq \gamma(2\alpha)^{-1}$ for all $k>k_0$, also assume that $n>k_0$. Now divide the function $f(\alpha,n)$ into three parts, so that $f(\alpha,n)=I_1+I_2+I_3$, where
\begin{align*}
    I_1&=\frac{2}{n}\sum_{k=1}^{k_0}\left(-1+n^{\alpha(2g_n-g_k-g_{n-k})}\right),\\
    I_2&=\frac{2}{n}\sum_{k=k_0}^{\lfloor n^{\gamma}\rfloor }\left(-1+n^{\alpha(2g_n-g_k-g_{n-k})}\right),\\
    I_3&=\frac{2}{n}\sum_{k=\lceil n^{\gamma}\rceil}^{n/2}\left(-1+n^{\alpha(2g_n-g_k-g_{n-k})}\right).
\end{align*} It remains to show that for every $\e>0$, exists a number $n_0$ such that for every $n>n_0$, then $I_i<\e$ for every $i\in\{1,2,3\}$.

\quad Concerning the term $I_1$, since $k_0$ is fixed, it follows that $I_1$ is a sum of $k_0+1$ elements. We have $k_0$ polynomials in $n$ in the form $n^{\lambda_k}$ for some $\lambda_k$, and one element of the form $k_0 n^{-1}$. Then, taking $n>2k_0$, observe that:
\begin{align*}
    \max_{k\leq k_0} \{\lambda_k\}&=\max_{k\leq k_0}\{\alpha(2g_n-g_k-g_{n-k})-1\}\\
    &\leq\alpha(1-\mu^{-1}+1-g_{n-k_0})-1\\
    &\leq \alpha(1-\mu^{-1})+\frac{\gamma}{2}-1.
\end{align*}Since $\gamma<2(1-\alpha(1-\mu^{-1}))$,  then $\max_{k\leq k_0} \{\lambda_k\}$ is negative. In particular, since each term goes to zero in $I_1$, it is possible to take $n_1=n_1(\e,\gamma,\alpha, k_0)$ such that $I_1<\e$ for every $n>n_1$.

For the term $I_2$, find $n_2'=n_2'(\gamma,\alpha)$ such that for every $n>n_2'$ it is true that $g_n-g_{n-k}=(1-g_{n-k})-(1-g_n)< \gamma(2\alpha)^{-1}$ for every $k\in[k_0,n^{\gamma}]$. In that way:
\begin{align*}
    \frac{2}{n}\sum_{k=k_0}^{\lfloor n^{\gamma}\rfloor}\left(-1+n^{\alpha(2g_n-g_k-g_{n-k})}\right)&\leq  \frac{2}{n}\sum_{k=k_0}^{\lfloor n^{\gamma}\rfloor} \exp\left\{\ln(n)\left(\alpha(g_n-g_k)+\frac{\gamma}{2} \right)\right\}\\
    &\leq 2\exp\left\{\ln(n)\left(\gamma+\frac{\gamma}{2}+\alpha(g_n-g_{k_0})-1\right)\right\} \\
    &\leq 2\exp\left\{\ln(n)\left(\frac{3\gamma}{2}+\alpha(1-g_{k_0})-1\right)\right\} \\
     &\leq 2\exp\left\{\ln(n)\left(2\gamma-1\right)\right\} 
\end{align*} Then, by the choice of $\gamma<1/2$ and $k_0$, find $n_2=n_2(\e,k_0,\gamma)>n_2'$ such that for every $n>n_2$ the value of $I_2$ satisfies $I_2<\e$.

 To compute $I_3$, first use that $(g_n)_n$ is a monotone sequence converging to $1$, then for $k\in[n^{\gamma},n]$, and $n$ is large enough:
\begin{align}\label{eq:Gumbel-:3}
    4\alpha (g_n-g_k) \ln{(n)}< \frac{4\alpha}{\gamma} (1-g_{n^{\gamma}}) \ln{(n^{\gamma})}.
\end{align}  Using that $(1+x)\geq e^{x/2}$ when $x<1$, we can conclude that
\begin{align*}
    (1+4\alpha(g_n-g_k)\ln{(n)})^{1/2\alpha(g_n-g_k)}\geq e^{\ln(n)}=n,
\end{align*} when $4\alpha(g_n-g_k)\ln{(n)}<1$. Therefore, for big values of $n$, it is true that:
\begin{align}\label{eq:gumbelappsimples}
    (-1+n^{2\alpha(g_n-g_k)})\leq 4\alpha(g_n-g_k) \ln{(n)}.
\end{align} In particular, by equation \eqref{eq:gumbelappsimples}, and since $g_k<g_{n-k}$ for every $k\in(n^{\gamma},n/2)$, we get: 
\begin{align*}
    I_3&=\frac{2}{n}\sum_{k=\lceil n^{\gamma}\rceil}^{n/2}\left(-1+n^{\alpha(2g_n-g_k-g_{n-k})}\right) \\
    &\leq \frac{2}{n}\sum_{k=\lceil n^{\gamma}\rceil}^{n/2}\left(-1+n^{2\alpha(g_n-g_k)}\right) \\
    &\leq \frac{8\alpha}{n}\sum_{k=\lceil n^{\gamma}\rceil}^{n/2} (g_n-g_k)\ln{n}
\end{align*}
 Using Lemma \ref{Lem:Gumbel-1} on \eqref{eq:Gumbel-:3}, for every fixed $\gamma$, take $n_3=n_3(\e,k_0,\gamma)$ such that $\frac{16\alpha^2}{\gamma}(1-g_{n^{\gamma}})\ln{(n^{\gamma})}<\e$; more than this, for every $n\geq n_3$:
\begin{align*}
    I_3\leq \frac{8\alpha}{n}\frac{n}{2}\frac{4\alpha}{\gamma} (1-g_{n^{\gamma}})\ln{(n^{\gamma})}<\e.
\end{align*} To finish the proof of the lemma take $n_0=\max\{n_1,n_2,n_3\}$.
\end{proof}

\subsubsection{Proof of Proposition \ref{Prop3}}
\label{subsubsection:ProofofProp3}

 All the above propositions do not use the full strength of the Theorem's hypotheses that $\Ex{R^{1+p}}<\infty$, for some $p>0$. But here, in Lemma \ref{Lem:Gumbel-2}, the necessity of this condition will become evident. 

 \begin{Rmk}
     Note that the random variable $R$ with distribution $f(r)=\frac{1}{n\ln^{3}(n)}$ satisfies equation \eqref{eq:Gumbel-4} from Lemma \ref{Lem:Gumbel-1}, but does not have any greater moment. In particular, for every $\lambda>0$,  $\lim_{k\to \infty} f(k)k^{1+\lambda}\ln{k}=\infty$, and $R$ does not satisfy the Lemma \ref{Lem:Gumbel-2}. Consequentially Proposition \ref{Prop3} is not true for it. Moreover, the covering of the remaining points, will not have a direct connection to the coupon collector problem.
 \end{Rmk}

\begin{Lem}\label{Lem:Gumbel-2}
Take $R$ which satisfies the hypothesis in equation \eqref{eq:ContinuousGumbelCondition}, this is, exists $\lambda>0$ such that $\lim_{k\to \infty} f(k)k^{1+\lambda}\ln{k}=0$. Then there exists a $\beta_0=\beta_0(\lambda)$, where for every $\beta>\beta_0$, $C>0$, and for every sequence of sets $(K_n)_n$, with $K_n\in \mathcal{K}_{\beta}^n$ for every $n\in \N$, we have:
\begin{align*}
\lim_{n\to \infty} \P{\bigcup_{k=1}^{N(Cn\ln{n})}\{|\mathcal{O}_k\cap K_n|\geq 2\} }=0,
\end{align*} where $N(t)$ is the Poisson process in the line with rate $1$, used in the definition of the continuous time covering process.
\end{Lem}
\begin{proof}[Proof of Lemma \ref{Lem:Gumbel-2}]
The probability that an object intercepts the set $K_n$ in two points or more is bounded by $f(n^{\beta})=\P{R>n^{\beta}}$. Therefore:
\begin{align*}
    \P{\bigcup_{k=1}^{N(Cn\ln{n})}\{|\mathcal{O}_k\cap K_n|\geq 2\} }\leq 1-e^{-Cf(n^{\beta})n\ln{n}}.
\end{align*} Now,  $\lambda$ satisfying the condition \eqref{eq:ContinuousGumbelCondition}, and $m=n^{(1+\lambda)^{-1}}$, it is true that:
\begin{align*}
    \lim_{n\to \infty} f(n^{(1+\lambda)^{-1}})n\ln{n}&=\lim_{m\to \infty} (1+\lambda)f(m)m^{1+\lambda}\ln{m}=0,
\end{align*} In this way, take $\beta>(1+\lambda)^{-1}$ to conclude the proof.
\end{proof}

To finish the proof, we need to understand  a simple connection with the Coupon collector. To state it, start by fixing a parameter $p\in(0,1)$, and a set $\{1,....,K\}$. Define a coupon collector of $\{1,....,K\}$ with a time change $p$, in the following way:  Consider a Poisson process with rate $1$, and for each point in the Poisson process, sample an independent Bernoulli with parameter $p$.  When the Bernoulli is equal to one, with probability $p$, take one of the possible $K$ points in the space uniformly. When the Bernoulli is equal to zero, with probability $(1-p)$,  do nothing. By the thinning argument of Poisson Point process, we can define $(\xi_k)_{k=1}^{K}$ as a set of independent exponential random variables with rate $\frac{p}{K}$, and, the time need to complete the space as: 
    \begin{align*}
        T_K^{\ell}=\max_{k=1,...,K}\{\xi_k\}.
    \end{align*} About this process, one may get that: 
\begin{Lem}\label{Lem:Gumbel-3}
Let $(Y_k)_k$ be a coupon collector of the set $\{1,...,K\}$ with time change $p\in (0,1]$, that may depend on $K$.  Then, set $T_K^{\ell}$ the time needed to take all the coupons, so
\begin{align*}
    \frac{p}{K}T_K^{\ell}-\ln{K}\overset{D}{\implies} \G(0,1)
\end{align*} when $|K|$ goes to infinity.
\end{Lem}

The proof of the Lemma \ref{Lem:Gumbel-3} is located in Subsection \ref{subsec:Usefulprop} of the Appendix. With this result we can finish the finally proof proposition \ref{Prop3}. 

\begin{proof}[Proof of Proposition \ref{Prop3}]  
For each object used in the covering $\mathcal{O}=\{U+\Gamma_R\}$, define the truncated object at height $n^{\beta}$ as $\overline{\mathcal{O}}=\{U+\Gamma_{\min\{R,n^{\beta}\}}\}$. With the truncated objects, consider the truncated covering as $\overline{X}_t=\bigcup_{k=1}^{N(t)} \overline{\mathcal{O}}_k$. 

Fix a sequence $(K(n))_n$, where $K(n)\in \mathcal{K}_{\beta}^n$ for every $n>0$. The set $K(n)$ is composed of disjoint and sparse points, therefore, by construction $\overline{X}_t$ behaves as a coupon collector of $K(n)$ with time change $\mu g_{n^{\beta}}/n$. So define:
\begin{align*}
    T_{K(n)}^{\ell}=\inf{ \left\{t: K\subset \overline{X}_t \right\}} .
\end{align*} By applying Lemma \ref{Lem:Gumbel-3}, one may get:
\begin{align}\label{eq:Proofprop3.1}
     \frac{\mu g_{n^{\beta}}}{n} T_{K(n)}^{\ell}-\ln{|K|}\overset{D}{\implies} \G(0,1). 
\end{align} To finish the proof we are going to replace $T_{K(n)}^{\ell}$ with $T_{K(n)}$, and remove the term $g_{n^{\beta}}$ in the equation \eqref{eq:Proofprop3.1}. 

To prove that $T_K^{\ell}$ is indeed $T_{K(n)}$ with high probability, we are going to use Lemma \ref{Lem:Gumbel-2} to show that no large object will appear in the time scale needed for the covering. More particular, define the event of having a big object until time $t$:
\begin{align*}
    E_t=\bigcup_{k=1}^{N(t)} \{ \mathcal{O}_k\neq \overline{\mathcal{O}}_k\}.
\end{align*} 

Assume $t=2n\ln{n}$ to be a suitable value of $t$. By Lemma \ref{Lem:Gumbel-2}, the events $E_t$ have probability going to zero, this is:
\begin{align*}
    \lim_{n\to \infty} \P{E_{2nlog(n)}}=0. 
\end{align*} 

Rest to show that $T_{K(n)}^{\ell}$ is lower than $2n\ln{n}$ with high probability. For this, consider the following bounds: $|K|<n$ and $g_{n^{\beta}}<1$. For big values of $n$ in the limit of equation \eqref{eq:Proofprop3.1}, the following is true:
\begin{align}\label{eq:Gumbelboundfornonlazyandlazy}
    \P{ T_{K(n)}^{\ell}> \frac{2n\log(n)}{\mu}}<1-\exp\{-1/n\}. 
\end{align} In particular, the probability limit of the event in question is the limit conditioned  on the event that happens with probability one, and that turns out to be bounded by the probability of large objects existing, $E_t$. That is:
\begin{align*}
    \lim_{n\to \infty} \P{T_{K(n)}^{\ell} = T_{K(n)}}&= \lim_{n\to \infty} \P{T_{K(n)}^{\ell} = T_{K(n)}|T_{K(n)}^{\ell}<2n\ln{n}}\\
    &\geq \lim_{n\to \infty} 1-\P{E_{2n\ln{n}}}=1.
\end{align*} 

Finally, to remove the term $g_{n^{\beta}}$ from equation \eqref{eq:Proofprop3.1} using the bound in equation \eqref{eq:Gumbelboundfornonlazyandlazy}, we get that $T_{K(n)}$ is of order $n\log(n)$ and therefore by the condition \eqref{eq:Gumbel-4}:
\begin{align*}
    \frac{\mu}{n}T_{K(n)}\left(1-g_{n^{\beta}}\right)\overset{P}{\to} 0,
\end{align*} concluding  that: 
\begin{align*}
     \frac{\mu }{n} T_{K(n)}-\ln{|K|}\overset{D}{\implies} \G(0,1). 
\end{align*} As desired.
\end{proof}


\section{Compact Support Phase}\label{sec:compactphase}\noindent

Unlike Section \ref{sec:Gumbel}, the theorem proved here is a more general version of the theorem stated in the Introduction. The additional conditions make the theorem more general, but less straightforward to understand. Therefore, we have intentionally postponed these conditions until now.

\begin{Theo}[Compact Support Phase]
Let $f\in \mathrm{RV}_{-1}$ that satisfies for all $\beta\in(0,1)$ that: 
\begin{align}\label{eq:Star}
    &\limsup_{n\to \infty} \sup_{n^{\beta} \leq x \leq n}\frac{xf(x)}{nf(n)}=b(\beta)<\infty, \text{ and }\\
    \label{eq:Triangle}
    &\limsup_{n\to \infty} \frac{\sum_{i=1}^{n^{\beta}}f(i)}{f(n)n\ln{n}}=d(\beta)<\infty.
\end{align} Then, $(T_n f(n))_n$ is tight. Moreover, for every subsequence $(n_k)_k$ such that
\begin{align} \label{eq:ContinuousCompactResult}
    f(n_k)T_{n_k}\overset{D}{\implies} Y.
\end{align} The distribution $Y=Y(f,(n_k)_k)$ is a non degenerated distribution with compact support.
\end{Theo}

\begin{Rmk}
Hypotheses \eqref{eq:Star} and \eqref{eq:Triangle} look strange at first glance. To make it tangible to the reader fix $\beta\in (0,1)$, and consider the following examples:
\begin{itemize}
    \item Take $f(x)=1/x$. Therefore, the value appearing in condition \eqref{eq:Star} is trivially equal to one. Using the harmonic series, we can calculate the value appearing in condition \eqref{eq:Triangle} which is equal to $\beta$.
    \item Let $b\in \R$ and take $f(x)=\frac{\ln^b{x}}{x}$. About condition \eqref{eq:Star}, we have that:
    \begin{align*}
         \sup_{n^{\beta}\leq x \leq n}\left\{\frac{xf(x)}{nf(n)}\right\}&=  \sup_{\gamma \in (\beta,1)}\left\{\frac{n^{\gamma}f(n^{\gamma})}{nf(n)}\right\}=\sup_{\gamma \in (\beta,1)}\left\{\gamma^{b}\right\}.
    \end{align*} For condition \eqref{eq:Triangle}, let $C$ be some constant, then we get that:
    \begin{align*}
        \sum_{i=2}^{n^{\beta}}f(i)&<\int_2^{n^{\beta}+1} f(x)dx\\
        &=\begin{cases}
            \frac{\ln^{b+1}{(n^{\beta})}}{b+1}+C, &\text{ where } b\neq -1.\\
            \ln{\ln{(n^{\beta})}}+C, &\text{ if } b=-1.
        \end{cases} 
    \end{align*} In particular, for all $b>-1$ the value in condition \eqref{eq:Triangle} if finite, and for all $b\leq -1$ the condition \eqref{eq:Triangle} is not satisfied.

    \item To see a case where the condition \eqref{eq:Star} is not satisfied, let $\gamma\in (0,1)$ and take $f(x)=\frac{\exp\{-\log^{\gamma}(x)\}}{x}$. We have that $f\in \mathrm{RV}_{-1}$, and:
    \begin{align*}
        \sup_{n^{\beta}\leq x \leq n }\frac{xf(x)}{nf(n)}> \frac{\exp\{\log^{\gamma}(n)\}}{\exp\{\log^{\gamma}(n^{\beta})\}}=  \exp\{(1-\beta^\gamma)\log^{\gamma}(n)\}.
    \end{align*} That explodes when $n$ goes to infinity for all $\beta\in(0,1)$ fixed.
\end{itemize}
\end{Rmk}
\bigskip 

In order to prove Theorem \ref{teo:2}, we divided the proof into two subsections. Then, the conclusion follows immediately by applying Prokhorov's Theorem to the sequence $(T_n f(n))_n$. More precisely, the proof follows the following steps:
\begin{enumerate}
    \item Subsection \ref{subsec:compactsupport} proves that $(T_nf(n))_n$ is tight, and that any limit in distribution belongs to some compact $[0,a^*]$ with $a^*>0$. 
    \item Subsection \ref{subsec:nontrivialdistribution} uses a technique to prove that the limit distribution is not degenerate, that is, a probability distribution with support only at a single point. 
\end{enumerate}

To simplify the proof, let us give another definition for the continuous covering process that will come in handy. Consider $S=\left(\mathbb{Z}/n\mathbb{Z}\right) \times \Z_+$, then define a Poisson Point Process $(\Omega,\mathcal{F},\mathbb{P})$ on $S$ with rate $\Lambda_{t}=\mathrm{Unif}(\mathbb{Z}/n\mathbb{Z})\otimes dR$. Where $\Omega=\{ \mathrm{w}: \mathrm{w}=\sum_{i\in I} \delta_{(u_i,r_i)}, $ $\text{ s.t. }(u_i,r_i)\in S \, \text{ for all } i\in I, I<\infty\}$ is the state space, and $\mathcal{F}$ is the smallest $\sigma-$algebra that makes the evaluation measures $\{\mathrm{w}(A): A\subset S\}$ measurable.

This Poisson process is not artificial; indeed if we place a point $(U_k,R_k)\in S$ for every object $\mathcal{O}_k=\{U_k+\Gamma_{R_k}\}$, then the points placed have the same distribution as a Poisson process with rate $\Lambda_t$. To see whether the points imply a covering, define the projection function as
\begin{align*}
    \Pi: S \to &\mathcal{P}(\mathbb{Z}/n\mathbb{Z})\\
    (u,r) \mapsto &\{u,u+1,...,u+r\}\in \mathbb{Z}/n\mathbb{Z}.
\end{align*} With the projection $\Pi$ defined, given any configuration $\mathrm{w}=\sum_{i\in I} \delta_{(u_i,r_i)}$, one can recover the covering process $X_t$ using the configuration $\mathrm{w}$ as:
\begin{align*}
    X_t=X_t(\mathrm{w})=\bigcup_{i\in I }\Pi((u_i,r_i)). 
\end{align*}

\subsection{Compact support} \label{subsec:compactsupport} 

In the Gumbel phase, the typical objects exhibit small sizes in comparison to the torus. Here, however, the presence of objects that are comparable in size to the space itself becomes significant. To control their number and what these big objects cover, we use a Branching Process argument that can be found in the following proposition:

\begin{Prop} \label{Prop:Compactsupport}
Let $f\in \mathrm{RV}_{-1}$ . Then there exists $a^*>0$ such that
\begin{align*}
    \lim_{n\to \infty} \P{T_nf(n)\leq a^*}=1.
\end{align*}
\end{Prop}

\begin{proof}[Proof of Proposition \ref{Prop:Compactsupport}] The proof of this proposition is bases on a comparison between the presence of large objects in the cover with a Branching Process. To properly define the comparison, we first need to define a set of regions in $S=\left(\mathbb{Z}/n\mathbb{Z}\right)\times \Z_+$, and a set of intervals on the torus. Then, through the vacant set, a relationship can be constructed between these objects. The essence of the proof lies on the following observation: If there is a vacant interval, and the region of object in $S$ capable of covering it entirely is empty, then we can divide the vacant interval into smaller sets, each set as child of the branching process will have a new independent chance to be cover or not by objects of another non explored region.  

Let $\mathcal{T}_4=(\mathbb{V}_4,\mathbb{E}_4)$ be a rooted tree in which every vertex has four decedents. Precisely, the set of vertex and edges are respectively:
\begin{align*}
    \mathbb{V}_4&=\left\{v(i,h)\middle|h\geq 0,\, i\in\{0,1,2,3\}^h\right\}, \text{ and}\\
    \mathbb{E}_4&=\left\{\left(v(i,h),v(i\times j, h+1)\right)\middle| v(i,h)\in \mathbb{V}_4,\, j\in \{0,1,2,3\}\right\}. 
\end{align*} For any vertex $v(i,h)\in \mathbb{V}_4$ where $i=(i_1,...,i_h)$, define its order as  $|i|=\sum_{j=1}^h i_j 4^j$, and set the following regions in $[0,n)\times (0,2n)$: 
\begin{align*}
    R(i,h)=\left[\frac{n|i|}{4^h},\frac{n(|i|+1)}{4^h}\right)\times\left[\frac{2n}{4^h},\frac{2n}{4^{h-1}}\right),
\end{align*} For the vertex $v(0,0)$, define the region $R(0,0)=\left[0,n\right)\times\left[2n,\infty\right)$. Associated with each region $R(i,h)$, define the interval $I(i,h)=\left[\frac{n|i|}{4^h},\frac{n(|i|+1)}{4^h}\right)\cap \Z$. See in Figure \ref{fig:2} a representation of the tree $\mathcal{T}_4$ side by side with the regions. 

\begin{figure}[htbp]
    \centering
    \includegraphics{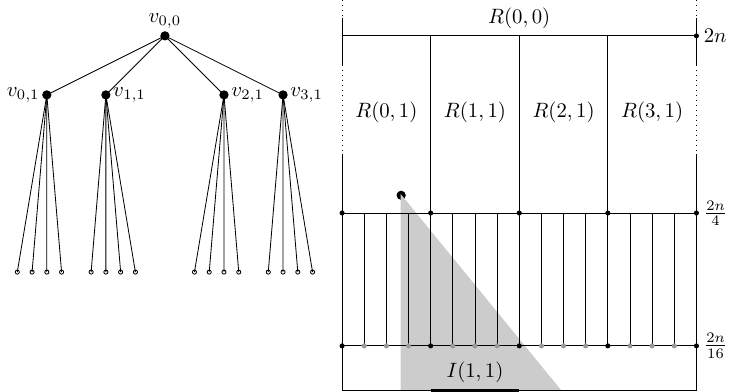}
    \caption{A representation side by side of the tree $\mathcal{T}_4$ and the set of intervals $R(i,k)$, on the rectangle $(0,1]\times (0,\infty)$.}
    \label{fig:2}
\end{figure}

Now, fix a covering process $X_t$ at time $t=\alpha f^{-1}(n)$, and along with it a configuration $\mathrm{w}$ with rate $\Lambda_t$. Then, we are going to define a heterogeneous branching process $(Z_h)_h$ in $\mathcal{T}_4$ using the regions $R(i,h)$, where $Z_0=1$, and associated with it, we have the vertex $v(0,0)$. For other values of $h$, define inductively
\begin{align*}
    Z_{h+1}= \sum_{i=1}^{Z_{h}} 4\cdot \ind\{\mathrm{w}(R(\left\{|v(i)|-1\right\}\mod{4^h},h)>0\},
\end{align*} where $\{v(1),...,v(Z_h)\}$ are the vertex associated to the $h-th$ generation of the branching process. Moreover, define the vertex associated for the next generation as the union of the four children of  $v\in \{v(i)\}_{i=1}^{Z_h}$ such that $\{\mathrm{w}(R(\{|v|-1\}\mod{4^h},h)>0\}$ happens.

The branching process $(Z_h)_h$, despite having a complex definition, was created to preserve one property: If it dies, it covers the space. In essence, notice two things: The intervals $I(i,h)$ fit inside $I(j,k)$ if and only if $v(i,h)$ is an ancestor of  of $v(j,k)$ (this is, $v(j,k)$ belongs to the unique path that connects $v(i,h)$ to the root of the tree); and, if $\{\mathrm{w}(R(i-1\mod{4^h},h))>0\}$ then $I(i,h)$ is completely covered by an object. In particular, if $\{Z_h=0\}$, each vertex in the tree $v(i,h)$ has a dead parent $v(j,k)$, or analogously, each interval $I(i,h)$ fits within some larger interval $I(j,k)$ such that $\{\mathrm{w}(R(j-1\mod{4^h},k))>0\}$. Therefore, $I(i,h)$ is covered for every $i$, and the space is fully covered by objects with a radius greater than $2n/4^h$.

The process is heterogeneous in probability, so we need caution when using classical branching arguments. To understand how the probability changes in each generation, fix a value of $n>0$.  Start by noticing that there is no object $\mathcal{O}$ of size smaller than one in the covering, then the regions $R(i,h)$ with $h>\left\lfloor\frac{\ln{2n}}{\ln{4}}\right\rfloor$ are always empty. About the dependence of the heterogeneous process, notice that by the Poisson construction, since the regions are disjoint, the survival probabilities despite being different are independent. Now, when $h\leq \left\lfloor\frac{\ln{2n}}{\ln{4}}\right\rfloor$, for every fixed region $R(i,h)$ by routine calculation:
\begin{align*}
    \P{\mathrm{w}(R(0,0))>0}&=1-\exp\left\{-\alpha\frac{f(2n)}{f(n)}\right\}\\
    \P{\mathrm{w}(R(i,h))>0}&=1-\exp\left\{-\frac{\alpha}{f(n)4^{h}}\left(f\left(\left\lceil\frac{2n}{4^{h-1}}\right\rceil\right)-f\left(\left\lfloor\frac{2n}{4^{h}}\right\rfloor\right)\right)\right\}.
\end{align*} To understand such values, since $f\in \mathrm{RV}_{-1}$, by the Representation Karamata's Theorem, see Proposition \ref{prop:RVresults} item $2$, it is true that
\begin{align*}
    f(r)=r^{-1}L(r),
\end{align*} where $L(r)$ is a slowly varying function. Then, for every fixed $h\geq 0$, by definition \eqref{eq:RV}, using the Proposition \ref{prop:RVresults} item $1$, we get that:
\begin{align}\label{eq:RV1prop1}
     \lim_{n\to \infty}& \frac{f(2n)}{f(n)}=\frac{1}{2}, \text{ and }\\
    \lim_{n\to \infty}& \frac{1}{{{4}}^{k}f(n)}\left(f\left(\frac{2n}{{{4}}^{k}}\right)-f\left(\frac{2n}{{{4}}^{k-1}}\right)\right)=\frac{3}{8} , \text{ and }\nonumber\\
    \lim_{n\to \infty}& \frac{1}{f(n)4^{h}}\left(f\left(\left\lceil\frac{2n}{4^{h-1}}\right\rceil\right)-f\left(\left\lfloor\frac{2n}{4^{h}}\right\rfloor\right)\right)=\frac{3}{8}\label{eq:RV2prop1}
\end{align} Therefore, the heterogeneous probability have a limit for each $h$ fixed using the equations \eqref{eq:RV1prop1} and \eqref{eq:RV2prop1} that relays just on the fact that $f\in \mathrm{RV}_{-1}$. In particular, taking $n$ to infinity, we get:
\begin{align}
    \lim_{n\to \infty} \P{\mathrm{w}(R(0,0))>0}&=1-e^{-\frac{\alpha}{2}}, \text{ and }\nonumber\\
    \lim_{n\to \infty} \P{\mathrm{w}(R(i,h))>0}&=1-e^{-3\alpha/8}.\label{eq:Limitprop2bound}
\end{align} 

Throughout the rest of the proof, define $\widehat{Z_t}$ a new homogeneous branching process that lives in the tree $\mathcal{T}_4$, and has all four children in one generation with probability equal to $e^{-\frac{2\alpha}{8}}$, that is, greater than the probability of both limits in the equation \eqref{eq:Limitprop2bound}. 
 
 About $\widehat{Z_t}$,  fix $\alpha^*(\widehat{Z_t})=4\ln{2}$, and notice that for every $\alpha>\alpha^*$ the branching process $\widehat{Z_t}$ in $\mathcal{T}_4$ dies almost surely.
 
 Now, fix $\alpha>\alpha^*$ and take any $\e>0$, set $h_0=h_0(\e,\alpha)$ such that 
 \begin{align*}
     \mathbb{P}_{\alpha}\left(\widehat{Z_h}=0, \text{ for some }h<h_0\right)>1-\e.  
 \end{align*}

  Finally, using this fixed value of $h=h(\e,\alpha)$ and the limits in equation \eqref{eq:Limitprop2bound}, find $n_0=n_0(h)$ such that for every $n>n_0$, the process $\widehat{Z}_h$ dominates the events $\ind\{\mathrm{w}(R(i,h))>0\}$ with $h<h_0$. In particular, remember that by construction, if the branching $Z_h$ dies, than the space is covered, therefore: 
\begin{align*}
    \P{T_nf(n)>\alpha}\geq \P{\widehat{Z_j}(\alpha)=0 \text{ for some }j\leq h_0}\geq 1-\e.
\end{align*}Taking the limit with $\e$ going to zero, one can conclude the theorem, for each $\alpha>\alpha^*=4\ln{2}$. 
\end{proof}

\subsection{Non degenerate distribution} \label{subsec:nontrivialdistribution} 

This section is devoted to the proof that $(f(n)T_n)_n$ has a nondegenerate limit. To prove this, it is sufficient to show that for small values of $\alpha$, we have:
\begin{align}\label{eq:ftnnotdegenerated}
    0<\liminf_{n\to \infty}\P{f(n)T_n>\alpha}\leq\limsup_{n\to \infty}\P{f(n)T_n>\alpha}<1.
\end{align} Different from Subsection \ref{subsec:compactsupport}, where a branching process technique shows that the space is covered by 
 objects with size comparable with the space; here, we need a more delicate approach that controls the full range of object sizes at the same time to say that the space is not covered.

Before going into more detail, let us prove the lower bound of equation \eqref{eq:ftnnotdegenerated}. Looking just at the objects greater than the space itself, we get that with $t=\alpha/f(n)$:
\begin{align*}
    \liminf_{n\to \infty} \P{f(n)T_n>\alpha}>\lim_{n\to \infty}\P{\mathrm{w}(\{r>n\})>0}=e^{-\alpha}. 
\end{align*} In particular, it is not zero.

The proof of the upper bound of equation \eqref{eq:ftnnotdegenerated} is divided in Propositions \ref{Prop:Bigobjectsnontrivialdistribution} and \ref{Prop:Smallobjectsnontrivialdistribution}. We start using a very similar branching argument from the proof of Proposition \ref{Prop:Compactsupport}, now for small values of $\alpha$. To define it, we will need a new tree and a new set of regions. 

  Let $\widehat{\mathcal{T}}=(\widehat{\mathbb{V}},\widehat{\mathbb{E}})\subset \mathcal{T}_4$ be a sub-graph where
\begin{align*}
    \widehat{\mathbb{V}}&=\left\{v(i,h)\middle|h\geq 0,\, i\in\{1,3\}^h\right\}, \text{ and}\\
    \widehat{\mathbb{E}}&=\left\{\left(v(i,h),v(i\times j, h+1)\right)\middle| v(i,h)\in \mathbb{V},\, j\in \{1,3\}\right\}. 
\end{align*} Observe the sub-graph  $\widehat{\mathcal{T}}$ of $\mathcal{T}_4$ in the Figure \ref{fig:3}. This restriction guarantees that any two siblings corresponds to separated intervals, and that distance will help with decoupling them. 

\begin{figure}[!ht]
    \centering
    \includegraphics{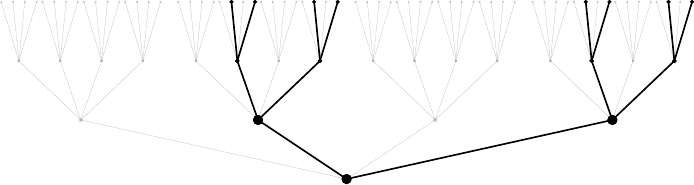}
     \caption{The black graph is a representative drawn of the sub graph $\hat{\mathcal{T}}$ inside the graph $\mathcal{T}_{{4}}$ drawn as white gray.}
    \label{fig:3}
\end{figure}
 
Using the graph $\widehat{\mathcal{T}}$, for every vertex $\widehat{v}(i,h)\in \widehat{\mathbb{V}}$ with $h>0$, define the region:
\begin{align*}
    \widehat{R}(i,h)=\left[\frac{n|i|}{4^h},\frac{n((|i|+2)}{4^h}\right)\times\left[\frac{n}{4^{h+1}},\frac{n}{4^{h}}\right).
\end{align*} For the vertex $v(0,0)$, define  $\widehat{R}(0,0)=\left[0,n\right)\times\left[\frac{n}{4},\infty\right)$. And, together with it, we set the intervals $\widehat{I}(i,h)=\left[\frac{n(|i|+1)}{4^h},\frac{n(|i|+2)}{4^h}\right)\subset  \mathbb{Z}/n\mathbb{Z}$. Observe such regions and intervals in Figure \ref{fig:4}. 

\begin{figure}[!ht]
    \centering
    \includegraphics{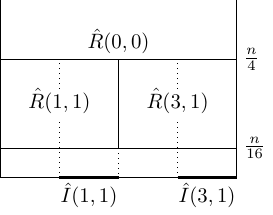}
    \caption{A representation of the regions $\hat{R}(i,h)$  and intervals $\hat{I}(i,h)$.}
    \label{fig:4}
\end{figure}

Fix now a covering process $X_t$ at time $t=\alpha f^{-1}(n)$, and together with it a configuration $\mathrm{w}$ with rate $\Lambda_t$. Define a heterogeneous branching process $(Z_h)_h$ in $\widehat{\mathcal{T}_4}$ using the regions $\widehat{R}(i,h)$, where $Z_0=1$, and associated with it, we have the vertex $v(0,0)$. For other values of $h$, define inductively
\begin{align*}
    Z_{h+1}= \sum_{i=1}^{Z_{h}} 2\cdot \ind\{\mathrm{w}(R(|v(i)|,h)=0\},
\end{align*} where $\{v(1),...,v(Z_h)\}$ are the vertex associated with the $h-th$ generation of the branching process. Moreover, define the vertex associated with the next generation as the union of the two children in $\widehat{\mathcal{T}_4}$ of each $v\in \{v(i)\}_{i=1}^{Z_h}$ such that $\{\mathrm{w}(R(|v|,h)=0\}$.  

To continue,, let us explore a relation between surviving in the branching process $(Z_h)_h$ and the covering process. Observe that $I(i,h)$ fits inside $I(j,k)$, if and only if $v(i,h)$ is an ancestor of  of $v(j,k)$. Also, notice that if $\{\mathrm{w}(\widehat{R}(i,h))=0\}$ then $I(i,h)$ does not intersect any object with radius between $n/4^{h+1}$ and $n/4^h$. Therefore, assuming that $\{Z_h>0\}$, each surviving vertex $v(i,h)$ has a family of survival ancestors (all vertex that belongs to the path that connects $v(i,h)$ to the root of the tree) . In other words, each interval $I(i,h)$ is not intersected by any object of size greater than $n/4^h$.

Unfortunately, the probability that the process survives in a generation $h$ does not behave well. The fact that the tail $f$ belongs to $\mathrm{RV}_{-1}$, controls only objects of size comparable to $n$, such as $n/4^h$ with $h$ fixed. In particular, to show that we do not cover the space, we need bounds to control objects of arbitrary size, as for example $n^{\beta}$ with $\beta\in (0,1)$. It is important to note that for some tails distributions, the rate of the regions $\widehat{R}(i,h)$ can explode as $h$ approaches $\left\lfloor\frac{\ln{n}}{\ln{4}}\right\rfloor$. 


To solve the heterogeneity problem, we divide the proof of the upper bound into two steps. First, we will explore the covering of objects with size greater than $n^{\beta}$. Second, we will work with the remaining objects, with sizes smaller than $n^{\beta}$. This division will be informally described by the terms big and small world.

More precisely, for any fix $\beta>0$, $t>0$, let $\mathrm{w}=\sum_{i\in I}\delta_{(u_i,r_i)}$ be any configuration of a Poisson Point Process in $S$ with rate $\Lambda_t$.  Define the small and big world to be respectively the processes:
\begin{align*}
X_t[1,n^{\beta})&=X_t[1,n^{\beta})(\mathrm{w})=\bigcup_{i\in I} \Pi((u_i,r_i))\ind\{r_i\in [1,n^{\beta})\}\\
X_t[n^{\beta},\infty)&=X_t[n^{\beta},\infty)(\mathrm{w})=\bigcup_{i\in I} \Pi((u_i,r_i))\ind\{r_i\in [n^{\beta},\infty)\}
\end{align*} Note that the covering  process $X_t$ can be written as $X_t[1,n^{\beta})\cup X_t[n^{\beta},\infty)$. And, by the thinning Poisson Theorem, we can work  with these two processes independently. 

The proof of the upper bound in relation \eqref{eq:ftnnotdegenerated} consists in taking $\alpha$ small and use Proposition \ref{Prop:Bigobjectsnontrivialdistribution} to find many empty regions of size $n^{\beta}$, using only the objects in the big world. Then, next in Proposition \ref{Prop:Smallobjectsnontrivialdistribution}, we will use a concentration bound to show that in one of theses vacant regions reveled by the big word,  we will have with high probability a empty set. In this way, the probability of covering the space does not approach one, when $\alpha$ is small. 

\begin{Prop} \label{Prop:Bigobjectsnontrivialdistribution}
Let $f$ be a distribution that satisfies condition \eqref{eq:Star} for some $\beta\in(0,1)$. Now, for all $\e>0$, and $\eta\in(0,1)$, we can find $\alpha_0=\alpha_0(\e,\beta,\eta)$ such that for all $\alpha<\alpha_0$, with probability greater than $1-\e$ there exist at least $n^{\frac{\eta}{2}(1-\beta)}$ intervals of size $\lfloor n^{\beta}\rfloor$ with mutual distance at least $\lfloor n^{\beta}\rfloor$ which are not intersected by the process $X_{\alpha/f(n)}[n^{\beta},\infty)$. 
\end{Prop}
\begin{proof}[Proof of Proposition \ref{Prop:Bigobjectsnontrivialdistribution}]  To prove this proposition, we will use condition \eqref{eq:Star} to create a bound on the branching process $\widehat{Z}_h$, for $h<\lceil\ln{n^{1-\beta}}/\ln{4}\rceil$. Next,  we will use the same method applied in the proof of Proposition \ref{Prop:Compactsupport} with a more specific branching theorem, witch regulates the number of children of the process. For this, take a vertex $\widehat{v}(i,h)\in\widehat{\mathbb{V}}$ and observe that:

\begin{align*}
    \P{\mathrm{w}(\widehat{R}(0,0))=0}&=\exp\left\{-\alpha\frac{f(n/4)}{f(n)}\right\}\\
    \P{\mathrm{w}(\widehat{R}(i,h))=0}&=\exp\left\{-\frac{2\alpha}{f(n)4^{h}}\left(f\left(\left\lceil\frac{n}{4^{h+1}}\right\rceil\right)-f\left(\left\lfloor\frac{n}{4^h}\right\rfloor\right)\right)\right\}.
\end{align*}

In particular, using condition \eqref{eq:Star}, we get that:
\begin{align*}
    \exp\left\{-\frac{2\alpha}{f(n)4^{h}}\left(f\left(\left\lceil\frac{n}{4^{h+1}}\right\rceil\right)- f\left(\left\lfloor\frac{n}{4^h}\right\rfloor\right)\right)\right\}
    &>\exp\left\{-\frac{2\alpha f\left(\left\lceil\frac{n}{4^{h+1}}\right\rceil\right)}{f(n)4^{h}}\right\}\\
    &= \exp\left\{-\frac{8\alpha \left\lceil\frac{n}{4^{h+1}}\right\rceil f\left(\left\lceil\frac{n}{4^{h+1}}\right\rceil\right)}{n f(n)}\right\}\\
    &>\exp\left\{-8\alpha b(\beta)\right\}
\end{align*} Then define $\widehat{Z}_h$ to be a homogeneous Branching process in $\widehat{\mathcal{T}}$ with parameter $e^{-8\alpha b}$ of having the two children. Observe that $\widehat{Z}_h$ dominates the branching process $Z_h$ for every $h\in\left[0,\lceil\ln{n^{1-\beta}}/\ln{4}\rceil\right)$, in that manner, we can couple both process in a way that if $\widehat{Z}_h$ survives, then $Z_h$ also survives.

Now, for every $\e>0$, and $\eta\in(0,1)$, there exists $\alpha_0=\alpha_0(\e,\eta,\beta)$ such that for all $\alpha<\alpha_0$ we have:
\begin{align*}
    \begin{cases}
        e^{-8\alpha b}> 2^{-1+\sqrt{\eta}}\\
        \P{\widehat{Z_j}=0, \text{ for some }j>0}\geq 1- \e/2.
    \end{cases}
\end{align*} In particular, with  probability close to one the process using just objects with size greater than $\lfloor n^{\beta}\rfloor$ does not cover some intervals of size $\lfloor n^{\beta}\rfloor$.  Finally, to know the number of such intervals we can use the concentration of the Branching process in \cite{AN}, stated for our case as
 \begin{Theo} \label{thm:Concentrationbranching}
 Let $X_{n,m}$ be independent and equally distributed positive integer random variables with $n>m>0$. Define the branching process as $Z_n=\sum_{m=1}^{Z_{n-1}} X_{n,m}$. If $\Ex{X_{1,1}}=\mu>1$, $\Var{X_{1,1}}=\sigma^2<\infty$ and $Z_0=1$, then exists a distribution $W$ such that:
 \begin{enumerate}
    \item $\frac{Z_n}{\mu^n}\to W$, almost sure. 
     \item $\lim\limits_{n\to \infty} \Ex{\left(\frac{Z_n}{\mu^n}-W\right)^2}=0$.
     \item $\Ex{W}=1$, and $\Var{W}=\frac{\sigma^2}{\mu^2-\mu}$.
     \item $\P{W=0}=q=\P{Z_n=0\text{ for some }n}$. 
 \end{enumerate}
 \end{Theo}
 
 Observe that the mean number of decedents in generation $h=\lceil\ln{n^{1-\beta}}/\ln{4}\rceil$ is equal to $(2e^{-8\alpha b})^h>n^{\frac{\eta}{2}(1-\beta)}$. By Theorem \ref{thm:Concentrationbranching}, since we are asking for the presence of significant less decedents than the mean, for every $\e>0$, exists a $n_0$ such that for $n>n_0$:
\begin{align*}
    \P{\widehat{Z}_h>n^{\frac{\eta}{2}(1-\beta)}}\geq 1-\e. 
\end{align*} This imply that, the process $\widehat{Z}_h$, that bounds bellow the number of intervals of size $\lfloor n^{\beta}\rfloor$ at height $h=\lceil\ln{n^{1-\beta}}/\ln{4}\rceil$ on the covering $X_t[n^{\beta},\infty)$, has more than $n^{\frac{\eta}{2}(1-\beta)}$ children with high probability. Since they have a distance between each other greater than $\lfloor n^{\beta}\rfloor$, the proof is finished using the coupling. 
\end{proof}

To finish the proof of the upper bound, let us use McDiarmid's concentration inequality, \cite{MC-Diarmid}, to give bounds over the small world when $\alpha$ is small.

\begin{Prop} \label{Prop:Smallobjectsnontrivialdistribution}
Let $f$ be a distribution that satisfies conditions \eqref{eq:Triangle} and \eqref{eq:Star} for some $\beta_0\in (0,1)$. Let $\eta\in(0,1/2)$, $\beta<\eta$. $\e>0$ and fix any set $\mathcal{I}$ formed by $n^{\eta}$ disjoint intervals of size $n^{\beta}$ that are spaced away from each other by at least $n^{\beta}$. Then there exists $\alpha'=\alpha'(\eta,\beta,\e)$ such that for all $\alpha<\alpha'$ with a probability greater than $1-\e$ the covering process $X_t[1,n^{\beta})$  does not cover $\mathcal{I}$ at time $t=\alpha/f(n)$.
\end{Prop}
\begin{proof} [Proof of proposition \ref{Prop:Smallobjectsnontrivialdistribution}]
Start by fixing $\eta\in (0,1/2)$ and $\beta<\eta$. Fix $\mathcal{I}=\mathcal{I}(\eta,\beta)$ to be an arbitrary set of $n^{\eta}$ intervals of size $n^{\beta}$, which are separated by a distance of at least $n^{\beta}$. Let $\{\mathcal{O}_1,\mathcal{O}_2,...,\mathcal{O}_N\}$ all objects placed in the covering process $X_t[1,n^{\beta})$ until time $t=\alpha/f(n)$ that intercept any point in $\mathcal{I}$. 

 Define a function $F(\mathcal{O}_1,\mathcal{O}_2,...,\mathcal{O}_N)$ that corresponds to  the total number of non covered point in the set $\mathcal{I}$ by the objects $\{\mathcal{O}_1,\mathcal{O}_2,...,\mathcal{O}_N\}$. Since the objects have size less than $n^{\beta}$ we have that for every $i<N$, it is true that $\Delta_i F\leq n^{\beta}$, where:
\begin{align*}
    \Delta_i F= \sup_{l_1,...,l_i,l_i',l_{i+1},...,l_N} |F(l_1,...,l_i,l_{i+1},...,l_N)-F(l_1,...,l_i',l_{i+1},...,l_N)|.
\end{align*}

Notice that the mean of the function $F$ is the average number of missing point at time $t=\alpha f^{-1}(n)$, so: 
\begin{align}
    \Ex{F(\mathcal{O}_1,...,\mathcal{O}_N)}=\mu_F&=n^{\eta+\beta}\P{0\notin X_t[1,n^\beta)} \nonumber\\
    &=n^{\eta+\beta}\exp\left\{-\frac{\alpha}{nf(n)}\left(-n^{\beta}f(n^{\beta})+\sum_{i=1}^{n^{\beta}} f(i)\right)\right\} \nonumber \\
    &= n^{\eta+\beta}\exp\left\{\frac{\alpha n^{\beta} f(n^{\beta})}{nf(n)}-\alpha \ln(n) \frac{\sum_{i=1}^{n^{\beta}} f(i)}{nf(n)\ln{n}}\right\} 
    \label{eq:Mean}.
\end{align}

  Observe that $N$ is a Poisson random variable with mean equal to $\lambda$, and using the fact that the intervals in the set $\mathcal{I}$ are disjoint by a distance of at least $n^{\beta}$, we have that.  
\begin{align}
    \lambda&= n^\eta \frac{\alpha}{f(n)}\left(\sum_{i=1}^{n^{\beta}} \frac{f(i)-f(n^{\beta})}{n} + \sum_{i=1}^{n^{\beta}} \frac{(1-f(n^{\beta}))}{n}\right) \nonumber \\
    &=\frac{\alpha n^{\eta+\beta}}{nf(n)}\left((1-2f(n^{\beta}))+ \frac{\sum_{i=1}^{n^{\beta}} f(i)}{n^{\beta}} \right) \label{eq:324}
\end{align}  

To continue the proof, let $d=d(\beta)\in\R$ from the hypotheses \eqref{eq:Triangle},  $\delta$ small, and consider that there exists some constants $C_1,C_2>0$ such that $\lambda\leq C_1 n^{\eta+\beta}$ and $\mu_F\geq  C_2 n^{\eta+\beta-\alpha (d+\delta)}$. Then, that for every $c>0$, we have:
 \begin{align*}
     \P{|F-\mu_F|>2c\mu_F}\leq &\P{|F-\mu_F|>2c\mu_F, N< \lambda(1+\lambda^{-1/3})}\\ &+\P{|N-\lambda|\geq \lambda^{2/3}} .
 \end{align*}Since, $N$ Poisson distributed with rate $\lambda$, and $\lambda$ diverges, by the Chebyshev's inequality we have that $\P{|N-\lambda|\geq \lambda^{2/3}}$ converges to zero. To deal with the other term, notice that since $N<\lambda(1+\lambda^{-1/3})<2C_1 n^{\eta+\beta}$ for big values of $n$, we get using McDiarmid's inequality that: 

\begin{align*}
     \P{|F-\mu_F|>2c\mu_F, N<\lambda(1+\lambda^{-1/3})}&\leq 2\exp\{-2c^2 {\mu_F}^{2}/ \lambda(1+\lambda^{-1/3})n^{2\beta}\}, \\
        &\leq 2\exp{\left\{- \frac{c^2C_2^2}{2C_1^2} n^{\eta-\beta-2\alpha (d+\delta) }\right\}},
\end{align*} In particular, choosing $\delta$ small enough, for any choice of $\eta$ and $\beta$, for every $\e>0$, one can find $\alpha'=\alpha'(\eta,\beta,\delta,\e)$ such that the probability stays below $\e$ for every $\alpha<\alpha'$, and therefore:
\begin{align*}
    \P{\{X_t[1,n^{\beta})]\}^c \cap \mathcal{I}\neq \emptyset}>1-\e,
\end{align*} for every choice of set $\mathcal{I}(\beta,\eta)$. 

To finish the proof, we just need to show that indeed exists constants $C_1,C_2>0$ such that $\lambda\leq C_1 n^{\eta+\beta}$, and $\mu_F\geq  C_2n^{\eta+\beta-\alpha (d+\delta)}$.

To show that there exists $C_1>0$ such that $\lambda\leq C_1 n^{\eta+\beta}$, notice that by the hypotheses \eqref{eq:Triangle}, we get for every $n>n_0$:
\begin{align*}
     \frac{\sum_{i=1}^{n^{\beta}} f(i)}{n^{\beta}nf(n)}< \frac{(d+\delta)\ln{n}}{n^{\beta}}.
\end{align*} In particular, in equation \eqref{eq:324},  we get that $\lambda\leq C_1n^{\eta+\beta}$.

To finish the proof, in the equation \eqref{eq:Mean} we need do two considerations, first find $n_0(\delta)$ such that for every $n>n_0$, by hypotheses \eqref{eq:Triangle}, we get that:
\begin{align*}
     \left|d-\frac{\sum_{i=1}^{n^{\beta}} f(i)}{nf(n)\ln{n}}\right|<\delta,
\end{align*} and use condition \eqref{eq:Star} to get that $\frac{n^{\beta} f(n^{\beta})}{nf(n)}$ is bounded in the limit. 
\end{proof}

To prove the upper bound on the equation \eqref{eq:ftnnotdegenerated}, we need to combine the Proposition \ref{Prop:Bigobjectsnontrivialdistribution} with the Proposition \ref{Prop:Smallobjectsnontrivialdistribution}. Take any $\eta\in (0,1)$, then find $\beta$ small such that $\frac{\eta}{2}(1-\beta)>\beta$, and that satisfies conditions \eqref{eq:Star} and \eqref{eq:Triangle}. Then, using Proposition \ref{Prop:Bigobjectsnontrivialdistribution}, find $\alpha_0(\e,\beta,\eta)$, such that the probability of surviving $n^{\frac{\eta}{2}(1-\beta)}$ disjoint intervals of size $n^{\beta}$ is at least $1-\e/2$. Next, since $\frac{\eta}{2}(1-\beta)>\beta$, using Proposition \ref{Prop:Smallobjectsnontrivialdistribution}, one can find a new $\alpha_1(\e,\beta,\frac{\eta}{2}(1-\beta))$, such that with probability greater than $1-\e/2$ in the set of surviving intervals of Proposition \ref{Prop:Bigobjectsnontrivialdistribution}, there exists an empty point. Therefore, the probability of not covering the space when $\alpha<\min\{\alpha_1,\alpha_0\}$ is greater than $1-\e$, and that concludes the upper bound. 

Also, the same proof allow us to inform that the limit distribution $Y$ has no atom in zero, that is: 

\begin{Cor}
    For every $f\in \mathrm{RV}_{-1}$, which satisfies conditions \eqref{eq:Triangle} and \eqref{eq:Star}. We have: 
    \begin{align*} 
    \lim_{\alpha\to 0}\limsup_{n\to \infty}\P{T_nf(n)<\alpha}=0. 
    \end{align*}
\end{Cor}

\section{Pre-exponential Phase}\label{sec:preexpphase}

Although it is possible to apply the same technique used in Section \ref{sec:compactphase} to prove Theorem \ref{teo:3}, we have decided to present an argument that introduces ideas for the next section. The ideas and techniques become more analytical and will need results about slowly varying functions presented in the appendix. 

Given any $p\in (-1,0)$ and $f\in \mathrm{RV}_p$, since $p$ different than $-1$, then by Karamata's theorem \ref{teo:Karamatastheorem} presented in the Appendix, one has that 
\begin{align} \label{eq:condpreexp}
    \lim_{n\to \infty} \frac{\sum_{i=1}^nf(i)}{nf(n)}=C_f>0.
\end{align}

The rest of the proof will show that for every $\alpha>0$, we have that:
\begin{align} \label{eq:Resultcontnpexp1}
    \limsup_{n\to \infty} \P{f(n)T_n\leq \alpha}&\leq 1-e^{-\alpha C_f},\text{ and }\\ 
    \liminf_{n\to \infty} \P{f(n)T_n\leq \alpha}&\geq 1-e^{-\alpha}+\left(1-\exp\left\{-\frac{\alpha}{2^{1+p}}\right\}\right)^2.\label{eq:Resultcontnpexp2}
\end{align} The proof of Theorem \ref{teo:3} follow by just applying Prokhorov's theorem, using the limits \eqref{eq:Resultcontnpexp1} and \eqref{eq:Resultcontnpexp2}.

The proof of \eqref{eq:Resultcontnpexp1} starts by observing that if $X_{\alpha/f(n)}$ does not cover the point $0$ at time $\alpha/f(n)$, then the space is not covered completely. So:  
\begin{align*}
    \P{f(n)T_n>\alpha}&\geq \P{0\notin X_{\alpha/f(n)}}\\
    &=\exp\left\{-\alpha \frac{\sum_{t=1}^n f(t)}{nf(n)}\right\}.
\end{align*} By the hypotheses \eqref{eq:condpreexp}, we have that 
\begin{align*}
   \liminf_{n\to \infty} \P{f(n)T_n>\alpha}&\geq \exp\left\{-\alpha C_f\right\}.
\end{align*} And, that is the upper bound in implication \eqref{eq:Resultcontnpexp1}.

\quad Now for the lower bound, define the following three disjoint regions in $\mathbb{Z}/n\mathbb{Z}\times \Z_+$. Similar to the branching process created in Section \ref{sec:compactphase}, those three regions correspond to the first children of the argument. 
\begin{align*}
   R_0&=\left\{(U,R):\, R\geq n\right\},\\
   R_1&=\left\{(U,R): U\in\left[0,\frac{n}{2}\right),\, R\in \left[\frac{n}{2},n\right)\right\},\\
   R_2&=\left\{(U,R): U\in\left[\frac{n}{2},n\right),\, R\in \left[\frac{n}{2},n\right)\right\}.
\end{align*} Notice that, we can cover the space if we hit an object in $R_0$, but also if $R_1$ and $R_2$ are occupied at the same time. Therefore: 
\begin{align*}
    \P{f(n)T_n\leq \alpha}&\geq \P{\mathrm{w}(R_0)>0}+ \P{\mathrm{w}(R_0)=0,\mathrm{w}(R_1)>0,\mathrm{w}(R_2)>0},
\end{align*} Now, using that $f\in \mathrm{RV}_p$, for $p\in (0,1)$, we can take the limit of this probabilities, as done in Proposition \ref{Prop:Compactsupport}. To get that: 
\begin{align*}
    \liminf_{n\to \infty} \P{f(n)T_n\leq \alpha}&\geq  1-e^{-\alpha}+e^{-\alpha}(1-e^{-\alpha/2^{1+p}})^2.
\end{align*} That finishing the proof of \eqref{eq:Resultcontnpexp2}.

\section{Exponential Phase}\label{sec:expphase}
\quad The proof of \ref{teo:4} is direct. We will show that with high probability the covering will occurs exactly when a big object appears. To show this, let  $f\in \mathrm{RV}_0$, then apply Karamata's theorem \ref{teo:Karamatastheorem} from the appendix to get that
\begin{align} \label{eq:continexponhypo}
    \lim_{r\to \infty} \frac{\sum_{i=1}^nf(i)}{nf(n)}=1.
\end{align}

Define the region $R_0=\{(U,R): R\geq n\}$, and observe that in time $\alpha/f(n)$, the number of objects in $R_0$, $N(R_0)$, is a Poisson random variable with rate $\alpha$. In this way, if we do not cover until time $\alpha/f(n)$ then $\{N(R_0)=0\}$ satisfies:
\begin{align*}
    \P{T_n>\frac{\alpha}{f(n)}}&\leq \P{N(R_0)=0}=e^{-\alpha}.
\end{align*} Now, if the point $0$ is not covered, naturally, we do not cover the entire space, so:
\begin{align*}
    \P{f(n)T_n>\alpha}&\geq \P{0\notin X_{\alpha/f(n)}}\\
    &=\exp\left\{-\alpha \frac{\sum_{t=1}^n f(t)}{nf(n)}\right\},
\end{align*} that by hypotheses \eqref{eq:continexponhypo}, we have that:
\begin{align*}
   \liminf_{n\to \infty} \P{f(n)T_n>\alpha}&\geq \exp\left\{-\alpha\right\}.
\end{align*} And, that concludes the convergence in distribution to the exponential random variable.

\section{Proof of Theorem \texorpdfstring{\ref{teo:2*} $\,$}{B* }} \label{sec:cylinder}
\noindent

In Section \ref{sec:compactphase}, a representation of the covering process in a cylinder $\mathbb{Z}/n\mathbb{Z}\times \Z$ appears in Propositions \ref{Prop:Compactsupport} and \ref{Prop:Bigobjectsnontrivialdistribution}, such representation makes evident several techniques used to understand the behavior of random coverage. After normalizing by $1/n$, the process in $\mathbb{Z}/n\mathbb{Z}\times \Z$ converges into a Poisson process across the continuous cylinder $\mathbb{S}^1\times[0,\infty)$ with a specific rate $\Lambda_\alpha$; however, it is important to note that the bare existence of such a limit is not sufficient to guarantee the convergence of the covering phenomena; indeed, the limit may deform or hide small objects essential to the covering in the discrete case, as discussed in Remark \ref{Rmk:SJ} for the Gumbel's phase.

This Section introduces and reviews some properties of a continuous covering model. The term continuous is used because we transition from working with a covering of the discrete torus $\mathbb{Z}/n\mathbb{Z}$ to performing a covering of the continuous circle $\mathbb{S}^1$.  It is important to clarify that the model and some of the results are not new in the literature; indeed, in 1972 a version of it was introduced by B.B.Mandelbrot in the seminal article \cite{BBM}, and in the same year the model was updated by L.A.Shepp in \cite{LAS}. Due to these two contributions, we will refer to it as the Mandelbrot-Shepp model. 

Both articles introduce the model on the real line, where the process has a regenerative property. In this paper, we focus on the circle; therefore, we need to adapt some of its proofs and definitions. For this reason, this section is divided into two parts.  Subsection \ref{subsec:Mandel-Brotdef}  aims to describe, define, and give some intuition behind the model on the circle. Subsection \ref{Subsec:proveoftheoremB*} is dedicated to solve the tightness issue of the limit, thus proving Theorem \ref{teo:2*}, which represents the novel contribution of this work showing that the discrete covering process converges in some sense to the Mandelbrot-Shepp model.

\subsection{The Mandelbrot-Shepp model} \label{subsec:Mandel-Brotdef}
\noindent

This Subsection defines the Mandelbrot-Shepp model and reviews some standard results about it. Although simple, such properties will be used directly or indirectly to prove Theorem \ref{teo:2*} in Subsection \ref{Subsec:proveoftheoremB*}.

 The Mandelbrot-Shepp model is defined as a covering process of the circle $\mathbb{S}^1=\R/ \mathbb{Z}$, represented here by the segment $[0,1)$. To define it, consider the cylinder $S=\mathbb{S}^1\times (0,\infty)$, fix $\alpha\geq 0$, and construct a Poisson point process in $S$ with rate $\Lambda_{\alpha}= \alpha dx\otimes\frac{dr}{r^2}$. Rigorously, consider the probability space $(\Omega,\mathcal{F},\mathbb{P}_{\alpha})$, where $\Omega=\{\omega=\sum_{i\in I} \delta_ {(x_i,y_i)}:  \,(x_i,y_i) \in S \, \forall i\in I, \text{ and } \omega(K)<\infty \,\forall K\subset S \text{ compact}\}$, and $\mathcal{F}$ is the smallest sigma algebra that makes the evaluation maps $\{\omega(K):K\text{ is compact in }S\}$ measurable.

 During the course of this article, configurations with different values of the parameter $\alpha$ will be compared. To simplify, define $\omega_{\alpha}$ as a configuration sampled with the measure $\Lambda_{\alpha}$, in particular, the parameter $\alpha$ is specified in the configuration notation. 
 
 To understand how different parameters interact, we should focus our attention on Lemma \ref{Lem:easycoupling}. The proof of the lemma is simple and was omitted from the paper. 
 
\begin{Lem} \label{Lem:easycoupling}
Fixed $\alpha,\beta>0$, and consider two independent configurations, that is, $\omega_{\alpha}=\sum_{i\in I} \delta_{(x_i,y_i)}$ and $\omega_{\beta}=\sum_{j\in J} \delta_{(x_j,y_j)}$ with intensities $\Lambda_{\alpha}$ and $\Lambda_{\beta}$ respectively.  Define: 
    \begin{align*}
        \omega_{\alpha}\cup \omega_{\beta}:= \sum_{k\in I\cup J} \delta_{(x_k,y_k)}.
    \end{align*} Then $\omega_{\alpha}\cup \omega_{\beta}$ has the same distribution as $\omega_{\alpha+\beta}$ with intensity $\Lambda_{\alpha+\beta}$.
\end{Lem}

  The Lemma \ref{Lem:easycoupling} allows us to think of the parameter $\alpha$ of the model as a time, and, as time passes, we place more and more objects on the cylinder $S$. Whenever we work with this type of construction between times $\alpha$ and $\beta$, we define the measure $\mathbb{P}=\mathbb{P}_{\alpha}\otimes \mathbb{P}_{\beta}$, where each configuration is independent of each other.  
 
 With the configurations defined, it is now possible to introduce the covering perspective of the process. Given a point $\xi=(x,y)\in S$, define the \textbf{Projection function} of $\xi$ as
\begin{align*}
   \Pi(\xi)=
    \begin{cases}
    [0,1), &\text{if } y>1.\\
    (x,x+y), &\text{if }y\leq1 \text{ and }x+y\leq 1.\\
    (x,1)\cup[0,x+y-1), &\text{if }y\leq1 \text{ and }x+y> 1.
    \end{cases}
\end{align*} Given any configuration  $\omega=\sum_{i\in I} \delta_{(x_i,y_i)}$ define:
\begin{align*}
    \mathcal{C}(\omega)&= \bigcup_{i\in I}\Pi( (x_i,y_i))\\
    \mathcal{V}(\omega) &=[0,1)\setminus \mathcal{C}(\omega),
\end{align*}to be respectively  the covered set  and  the vacant set of the Mandelbrot-Shepp model.  Whenever the configuration $\omega$ is fixed, or when no confusion arises, we denote by $\mathcal{V}$ and $\mathcal{C}$ those random sets. 
 
\begin{Lem} \label{Lem:coverpinD}
Given any parameter $\alpha>0$, and any point $z\in[0,1)$. Then, $\Pa{z\in \mathcal{C}}=~1$.  
\end{Lem}
\begin{proof}[Proof of Lemma \ref{Lem:coverpinD}]
Let $z\in [0,1)$ and define the region $R_z=\{x\in S: z\in \Pi(x)\}$. Computing the intensity of the Poisson process in $R_z$, we have:
\begin{align*}
    \Lambda_{\alpha}(R_z)=\int_{R_z} \frac{\alpha}{y^2} dydx\geq \int_{0}^z\int_{z-x}^{\infty} \frac{\alpha}{y^2} dydx= \int_{0}^z\frac{\alpha}{z-x}dx=\infty.
\end{align*} Therefore, for every $\alpha>0$ the event $\{\omega_{\alpha}(R_z)>1\}$ happens almost surely. In particular, $z$ is covered almost surely, concluding the proof.
\end{proof}

\begin{Cor}\label{Cor:noracional}
Given any parameter $\alpha>0$, and $\mathbb{Q}$ any enumerable set of points in $[0,1)$, we have
$\Pa{\mathbb{Q}\subset \mathcal{C}}=1.$
\end{Cor}

As observed by Corollary \ref{Cor:noracional}, any enumerable set is almost surely covered by $\mathcal{C}(\omega)$. This observation might lead us to believe that the model is always fully covered. However, we need to be cautious in drawing such conclusions, since the circle $\mathbb{S}^1$ is not countable.

To show that the Mandelbrot-Shepp model presents a non-trivial covering, an argument similar to Proposition~\ref{Prop:Bigobjectsnontrivialdistribution} can be used, see the following Lemma:
\begin{Lem}\label{Lem:nontrivialcovering2}
    Given $\alpha< \frac{\ln(2)}{6}$, then $\Pa{\mathcal{V}\neq\emptyset}>0$.
\end{Lem}

The proof that the model is indeed a non-trivial covering is not new, and can be seen in \cite{BBM,LAS}. In order to make the argument complete, one can find in Appendix \ref{subsec:Usefulprop} the proof of the Lemma~\ref{Lem:nontrivialcovering2} based on the branching process technique used in Proposition \ref{Prop:Bigobjectsnontrivialdistribution}. Moreover,  in Proposition \ref{prop:sat} a better bound on the covering probability will be proved based on the Shepp seminal paper \cite{Sh}.

 \bigskip

The Mandelbrot-Shepp model is a non-trivial continuous covering process that uses infinitely many objects, whereas the discrete covering process uses only a finite number. The link between these models will be established through a finite truncated version of the Mandelbrot-Shepp model. Consequently, to prove Theorem \ref{teo:2*}, we must establish two crucial connections: first, the relationship between the Mandelbrot-Shepp model and its truncated version, and second, the connection between the latter and the discrete model.

Given a configuration $\omega=\sum_{i\in I} \delta_{(x_i,y_i)}$ and any real number $z>0$,  define the \textbf{truncated  configuration at height z} as
\begin{align}\label{eq:Truncatedconfiguraiton}
 \omega[z]=\sum_{i\in I} \delta_{(x_i,y_i)}\ind{\{y_i>z\} }.
\end{align} In essence, the configuration $\omega[z]$ is given by the points with height greater than $z$. The next lemma connects such configurations to the un-truncated model. 

\begin{Lem}\label{lem:Topology}
    For any parameter $\alpha>0$, we have 
    \begin{align*}
        \Pa{\mathcal{V}\neq \emptyset}=\Pa{\bigcap_{n=1}^{\infty} \left\{\mathcal{V}\left(\omega\left[\frac{1}{n}\right]\right)\neq \emptyset\right\}}.
    \end{align*}
\end{Lem}
\begin{proof}[Proof of Lemma \ref{lem:Topology}]
Fixed any $n>0$, note that: 
\begin{align*}
    \{\mathcal{V}\neq \emptyset\}\subseteq \left\{\mathcal{V}\left(\omega\left[\frac{1}{n}\right]\right)\neq \emptyset\right\}.
\end{align*} Therefore: 
\begin{align}\label{eq:Equalityinprob}
     \Pa{\mathcal{V}\neq \emptyset}\leq \Pa{\bigcap_{n=1}^{\infty} \left\{\mathcal{V}\left(\omega\left[\frac{1}{n}\right]\right)\neq \emptyset\right\}}.
\end{align}To prove the opposite inequality, we use a topological argument. Observe that the set of points in $\omega_{\alpha}\left[\frac{1}{n}\right]$ is almost surely finite. Moreover, since the projection function of any point is an open set, we have that $\mathcal{V}_{\alpha}\left(\omega\left[\frac{1}{n} \right]\right)$ is the complementary of a finite union of open set, thus it is almost surely closed in $\mathbb{S}^1$, i.e. compact. Finally, consider $m>n$, and notice that $\mathcal{V}\left(\omega\left[\frac{1}{m}\right]\right)\subset \mathcal{V}\left(\omega\left[\frac{1}{n}\right]\right)$, in particular, $\bigcap_n \mathcal{V}\left(\omega\left[\frac{1}{n}\right]\right)$  is the intersection of nested, compact sets of $\mathbb{S}^1$. Therefore, if all are non empty, there must be a point in the limit, and the space will not be covered at time $\alpha$. Proving then the equality in~\eqref{eq:Equalityinprob}.    
\end{proof}

Since the Mandelbrot-Shepp model has a non-trivial covering phenomenon, it is interesting to define the \textbf{Cover function} of the the space as
\begin{align*}
    \pi(\alpha)=\Pa{\mathcal{C}=[0,1)}.
\end{align*} Also, set \textbf{the Cover function at height $z$} as  
\begin{align*}
    \pi_z(\alpha)=\Pa{\mathcal{C}\left(\omega[z]\right)=[0,1)}.
\end{align*} 

 As a consequence of the Lemma \ref{lem:Topology}, the link between the truncated version and the Mandelbrot-Shepp model can be created. Note the following. 

\begin{Lem} \label{lem:continuityofprobabilitytruncated}
  For every $z>0$,  $\pi_z(\alpha)$ is a continuous function in $\alpha$ and $\lim\limits_{n\to \infty} \pi_{\frac{1}{n}}(\alpha)=\pi(\alpha)$.  
\end{Lem}
\begin{proof}[Proof of Lemma \ref{lem:continuityofprobabilitytruncated}]
To prove that $\pi_z(\alpha)$ is a continuous function for every fixed $z>0$, define the region $R_z=[0,1)\times[z,\infty)$. Then:
\begin{align*}
    \Lambda_{\alpha}(R_z)=\alpha \int_0^1\int_{z}^{\infty}\frac{dydx}{y^2}=\frac{\alpha}{z}.
\end{align*} In particular, for every $\e>0$, by Lemma \ref{Lem:easycoupling}, one gets:
\begin{align*}
    |\pi_z(\alpha+\e)-\pi_z(\alpha)|&=\P{\mathcal{V}\left(\omega_{\alpha}[z]\right)\neq \emptyset, \mathcal{C}\left(\omega_{\alpha+\e}[z]\right)=[0,1)}\\
    &\leq \mathbb{P}_{\e}\left(\omega(R_m)>0\right)=1-e^{-\e/m}.
\end{align*} Implying that for every fixed $z>0$ the function $\pi_z(\alpha)$ is a right continuous function in $\alpha$. The proof of left continuity is analogous.  

To prove that $\lim_{n\to \infty}\pi_{\frac{1}{n}}(\alpha)=\pi(\alpha)$, it is sufficient to show that for any $\alpha>0$:
\begin{align*}
    \lim_{n\to \infty} \pi(\alpha)-\pi_{\frac{1}{n}}(\alpha)=0,
\end{align*} which is equivalent to
\begin{align*}
    \lim_{n\to \infty} \Pa{\mathcal{C}(\omega)=[0,1),\mathcal{V}\left(\omega\left[\frac{1}{n}\right]\right)\neq\emptyset}=0.
\end{align*} To finish the proof, note that by inclusion of the events, the limit is the intersection of them. In particular, it was already shown by  Lemma \ref{lem:Topology}, that the limit has zero probability, as we desired. 
\end{proof}

Thus, the Mandelbrot-Shepp process has a non-trivial covering function $\pi$ that can be studied by the approximation through the truncated versions $\pi_z$. Furthermore, by approximation, the vacant set $\mathcal{V}$ is also a limit of the vacant sets in $\omega[z]$. Then, to finish the proof, one will need to construct a link between the truncated spaces and the discrete model.

\subsection{Proof of Theorem \texorpdfstring{\ref{teo:2*} $\,$}{B* }}\label{Subsec:proveoftheoremB*}
\noindent

 The goal of this subsection is to prove Theorem \ref{teo:2*}. For this, we need to prove two points: The support of the distribution of covering in the Mandelbrot-Shepp model is $[0,1]$, and the limit of the covering probability in the discrete model converges to the function $\pi$. Since the proof of these points presents different arguments, we divide this subsection into two. In the first part, Subsection~\ref{Subsubsec:picompact}, the goal is to show that the function $\pi$ is not trivial in $[0,1]$, and equal to one above one. In the second part, Subsection~\ref{Subsubsec:pilimit}, the goal will be to find the limit of the vacant sets and the discrete covering probabilities. As a direct consequence of all the propositions presented here, Theorem \ref{teo:2*} will be derived. 

\subsubsection{Support of the function \texorpdfstring{$\pi$}{n}}\label{Subsubsec:picompact} \noindent

In $1956$, Dvoretzky in \cite{Dv} proposed another problem in the context of covering, which we now introduce. First, fix the space as $\mathbb{S}^1$, the circle with unit length, and fix a decreasing sequence $(\ell_n)_n$. At each time $k$, one samples a uniform point in the circle and places an arc starting at this point with length $\ell_k$. It was shown that if $\sum_n \ell_n=\infty$ each point in $\mathbb{S}^1$ is covered with probability one, but not necessarily $\P{\mathbb{S}^1 \text{ is fully covered}}=1$. Later, Shepp showed the necessary and sufficient condition in $1972$, in \cite{Sh}, described by:

\begin{Theo}[Shepp] \label{teo:Sheep}
Let $0<\ell_{n+1}\leq \ell_n\leq ...\leq \ell_2\leq \ell_1<1$, $n=1,2,...$, be arcs that are placed independently and uniformly on a circumference $\mathbb{S}^1$ of unit length. The union of these arcs covers $\mathbb{S}^1$ with probability one if and only if
\begin{align*}
    \sum_{n=1}^{\infty} n^{-2}\exp\{\ell_1+...+\ell_{n}\}=\infty.
\end{align*}
\end{Theo}

The articles \cite{BBM,LAS} exposed that $\{\alpha=1\}$ is a threshold for the Mandelbrot-Shepp model in the real line, where conditioning on the origin not being covered, a non-trivial set of vacant objects appears with positive probability.  Here in the circle, this value holds the same significance, which makes the result of Proposition \ref{prop:sat} not surprising. There are many ways to proceed with the proof of Proposition \ref{prop:sat}, we choose to use a concentration bound on the Poisson random variables and Theorem \ref{teo:Sheep}. This approach establishes an explicit connection between the Dvoretzky problem and our model.
\bigskip

\begin{Prop}\label{prop:sat}
 For the Mandelbrot-Shepp model $\pi(\alpha)=1$ for all $\alpha>1$ and $\pi(\alpha)<1$ for all $\alpha<1$.
\end{Prop}
\begin{proof}[Proof of Proposition \ref{prop:sat}]
Theorem \ref{teo:Sheep} does not allow $(\ell_n)_n$ to assume random values, but this problem can be resolved by conditioning.  To start, let $\omega=\sum_{i\in I} \delta_{(x_i,y_i)}$ be a configuration, and define the random sequence $(\ell_n)_n$, where:
\begin{align*}
    \ell_n=\sup\{r>0: \omega\{y\geq r\}=n\}. 
\end{align*} That is, the size of the $n-th$ biggest object.

Together with the sequence $(\ell_n)_n$, define the regions where objects are expected to belong:
\begin{align*}
    R_n^\delta=[0,1)\times \left[\frac{\delta}{n},\infty\right).
\end{align*}

With $\e>0$ small and $\delta=1-\e^2$, notice that $\omega_{1+\e}(R_{n}^1)$ is a Poisson random variable with rate $(1+\e)n$, and $\omega_{1-\e}(R_n^{\delta})$ is a Poisson random variable with rate $ n/(1+\e)$. As a application of the Chernoff bound for the Poisson random variable, it is possible to prove the following Lemma; the proof of which is postponed to the Appendix~\ref{subsec:Usefulprop}.  
\begin{Lem} \label{lem:poissoncontrol}
 For any $\e>0$, letting $\delta=1-\e^2$, we have that
\begin{align}\label{eq:lempoissoncontrol1}
    \sum_{n=1}^{\infty} \mathbb{P}_{1+\epsilon}\left({\omega(R_n^1)<n}\right)<\infty, \text{ and } \\
    \sum_{n=1}^{\infty} \mathbb{P}_{1-\epsilon}\left({\omega(R_n^{\delta})> n}\right) <\infty\label{eq:lempoissoncontrol2}. 
\end{align} 
\end{Lem}

Recall that we need to show two things: First, when $\alpha>1$ we have $\pi(\alpha)=1$. Second, when $\alpha<1$ then $\pi(\alpha)<1$. 

Start by fixing $\alpha>1$, and observe that: 
\begin{align*}
    \{\omega(R_n^1)<n\}=\left\{\ell_{n}<\frac{1}{n}\right\}.
\end{align*} Using equation \eqref{eq:lempoissoncontrol1}, together with Borel Cantelli we conclude that:
\begin{align}\label{eq:covering1prob}
    \P{\exists n_0, \text{ s.t. }\ell_n>\frac{1}{n} \forall n>n_0}=1.
\end{align}

Fixed the sequence $(r_n)_n$ where $r_n=1/n$ when $n>n_0$ and zero otherwise. There exists a constant $C=C(n_0)$ such that:
\begin{align*}
   \sum_{n=1}^{\infty} n^{-2}e^{r_1+...+r_n}&= \sum_{n=1}^{\infty} n^{-2}e^{r_1+...+r_{n_0}} \exp\left\{\sum_{k=n_0}^n r_k\right\}
   \geq \sum_{n=1}^{\infty} n^{-2}\exp\left\{\sum_{k=n_0}^n \frac{1}{k}\right\} \geq \sum_{n=1}^{\infty} \frac{C}{n}=\infty.
\end{align*} In other words, by Theorem~\ref{teo:Sheep}, for all $n_0>0$, the sequence $(r_n)_n$ with $n>n_0$ covers the circle with probability one. So if $\alpha>1$, by equation 
\eqref{eq:covering1prob}, one gets the following. 
\begin{align*}
    \pi(\alpha)=\sum_{k=1}^{\infty} \P{\mathcal{C}(\omega)=[0,1), \left\{\ell_n>\frac{1}{n}\, \forall n>k\right\}}.
\end{align*} And, if instead of placing objects with size $\ell_n$ we place an smaller object with fixed size $r_n=1/n$, by Theorem \ref{teo:Sheep} the space is going to be fully covered. So, by coupling one object into another $\pi(\alpha)=1$ for every $\alpha>1$ as desired. 

Analogously, for any $\e>0$ small, and $\alpha=1-\e<1$ take $\delta=1-\e^2$, and notice that:
\begin{align*}
    \{\omega(R_n^{\delta})>n\}=\left\{\ell_{n}>\frac{\delta}{n}\right\}.
\end{align*} Using the equation \eqref{eq:lempoissoncontrol2} together with Borel Cantelli we conclude that there exists just a finite number of regions $R_{n}^{\delta}$ with more than $n$ objects, thus:
\begin{align*}
    \P{\exists n_0, \text{ s.t. }\ell_n<\frac{\delta}{n} \forall n>n_0}=1.
\end{align*}

Fixed a sequence $(r_n)_n$ where $r_n<\frac{\delta}{n}$ for every $n>n_0$, then there exists a constant $c=c(n_0)$ such that:
\begin{align*}
   \sum_n n^{-2}e^{r_1+...+r_n}&
   \leq e^{n_0}+ \sum_{n>n_0} n^{-2}\exp\left\{\delta \sum_{k=n_0}^n \frac{1}{k}\right\} <e^{n_0}+ \sum_n c n^{-1-\e^2}< \infty. 
\end{align*} In other words, by Theorem \ref{teo:Sheep} the sequence $(r_n)_n$, where $r_n<\frac{\delta}{n}$ for every $n>n_0$, does not cover the space with probability one. In particular, if $\alpha<1$, there exists $n_0>0$ such that:  
\begin{align*}
    1-\pi(\alpha)> \P{\mathcal{V}(\omega)\neq \emptyset, \left\{\ell_n<\frac{\delta}{n}, \forall n>n_0\right\}}>0.
\end{align*} Where, if instead of placing an object $\ell_n$, we place an bigger object of size $\delta/n$. Then by Theorem \ref{teo:Sheep} the space have positive probability to not be covered. So, $\pi(\alpha)<1$ for every $\alpha<1$, as desired. 
\end{proof}

\subsubsection{Limits in distribution}\label{Subsubsec:pilimit} \noindent

 In order to prove the limit of the discrete process towards the Mandelbrot-Shepp model, we will create a coupling between its truncated version and the discrete model.  Our objective with the coupling is to demonstrate that in the limit, whenever the truncated version covers the space, the discrete model also covers it, and vice verse. This association between the truncated version and the discrete model, together with the limit in Lemma \ref{lem:continuityofprobabilitytruncated}, will give us the desired distributional limit.

We utilize the graphical construction of the Mandelbrot-Shepp model to simultaneously construct the discrete covering process. This construction establishes a monotonic coupling between both systems, indicating that covering in the Mandelbrot-Shepp model implies covering in the discrete model. However, it should be noted that in this coupling the converse is not necessarily true: it is possible that while the Mandelbrot-Shepp model is not covered, the associated discrete model is. To address this issue, we need a quantitative connection between the two models.

Given a configuration $\omega=\sum_{i\in I} \delta_{(x_i,y_i)}$ in the cylinder $S$, we define the process $W^n=W^n(\omega)$ in the torus $\mathbb{Z}/n\mathbb{Z}$ (a covering process with radius distribution $\P{R>r}=1/r$), as follows: for every $\ell \in\{0,...,n-1\}$, and $k\in \N$, define the regions in $[0,1)\times [1/n,\infty)$:
\begin{align*}
    \widehat{R}_{\mathcal{O}(\ell ,k)}= \left[\frac{\ell}{n},\frac{\ell+1}{n}\right)\times  \left[\frac{k}{n},\frac{k+1}{n}\right). 
\end{align*} Using such regions, define the process $W^n=W^n(\omega)$ as:
\begin{align*}
    W^n=W^n(\omega)=\bigcup_{\ell =0}^{n-1} \bigcup_{k=1}^{\infty} \mathcal{O}(\ell,k) \ind\left\{\omega\left(\widehat{R}_{\mathcal{O}(\ell ,k)}\right)>0\right\},
\end{align*} where $\mathcal{O}(\ell ,k)=\{\ell,\ell+1,\cdots, \ell+k-1\}$ is an arc. 

Notice that since
\begin{align*}
    \P{\omega\left(R_{\mathcal{O}(\ell,k)}\right)=0}=\exp\left\{-\frac{\alpha}{n}\left(\frac{1}{k}-\frac{1}{k+1}\right)\right\},
\end{align*} the process $W^n(\omega)$ has the same distribution as the continuous covering process with radius distribution $\P{R>r}=1/r$ at time $\alpha n$. So, we couple both process in the natural way.  

Observe the following problem thought. It is possible that the process $W^n$ covers $\mathbb{Z}/n\mathbb{Z}$, while in the truncated representation $\omega[1/n]$, there exist points in the discrete set $\left\{\frac{\ell}{n}, \ell\in \{0,1,\cdots,n\}\right\}\subset[0,1)$ not covered. As a result, providing information on the process $W^n$ relying solely on the configuration $\omega[1/n]$ is a complex task.  

Given this challenge, we introduce another discrete auxiliary covering process, which will help us to demonstrate that the covering probability converges to $\pi$.  Moreover, with this auxiliary covering process, one will be able to relate it back to the process $W^n$ proving then Theorem \ref{teo:2*}.

Analogously to $W^n$, define the process $X^n=X^n(\omega)$ in the torus $\mathbb{Z}/n\mathbb{Z}$ (a covering process with radius distribution $\P{R>r}=\log(1+1/(r-1))$) as follows: For every $\ell \in\{0,...,n-1\}$, and $k\in \N$, define the following regions in $[0,1)\times [1/n,\infty)$:
\begin{align*}
    R_{\mathcal{O}(\ell ,k)}=\left\{x\in S:  \left\{\frac{\ell }{n},\frac{\ell +1}{n},...,\frac{\ell +k-1}{n}\right\}\subset \Pi(x)\right\}. 
\end{align*} As the name suggests, if $\{\omega(R_{\mathcal{O}(\ell ,k)})>0\}$ the covering $X^n$ will have an object $\mathcal{O}(\ell ,k)=\{\ell ,\ell +1,...,\ell +k-1\}$. More precisely, define $X^n$ as:
\begin{align}\label{eq:defXlog}
    X^n=X^n(\omega)=\bigcup_{\ell =0}^{n-1} \bigcup_{k=1}^{\infty} \mathcal{O}(\ell, k) \ind\left\{\omega\left(R_{\mathcal{O}(\ell ,k)}\right)>0\right\}.
\end{align} Observe in Figure~\ref{fig:5}, the regions related to both processes $W^n$ and $X^n$. 

\begin{figure}[ht!]
    \centering \includegraphics{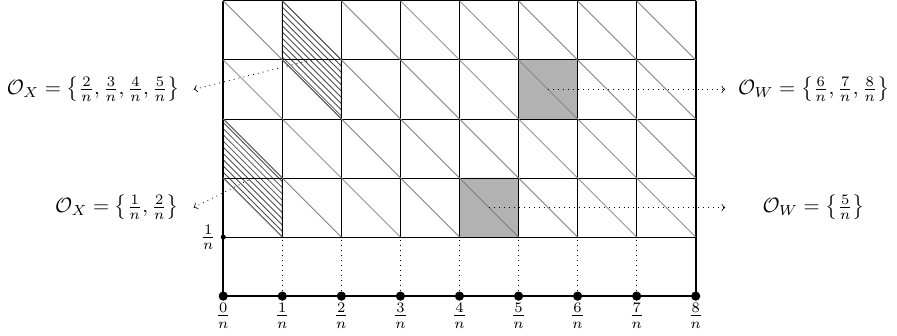}
    \caption{In the figure the regions corresponding respectively to the objects in the process $W$ and $X$ are drawn using gray and a pattern of lines respectively. Each square region for $W$ corresponds to an object, and each lozangular region for $X$ corresponds with another object. Notice that given a realization of $\omega[1/n]$, in the coupling looking to the set $P_n=\left\{\frac{\ell}{n}: \ell \in\{0,1,\cdots,n\right\}$, every point covered by $W^n$ is also covered by $X^n$.}
    \label{fig:5}
\end{figure}

Concerning $X^n$, we have that:
\begin{align*}
    \P{\omega_{\alpha}\left(R_{\mathcal{O}(\ell ,k)}\right)=0}=\exp\left\{-\frac{\alpha}{n}\left(\log\left(1+\frac{1}{k-1}\right)-\log\left(1+\frac{1}{k}\right)\right)\right\}.
\end{align*} Thus $X^n(\omega_{\alpha})$ has the same distribution as the continuous covering process with radius distribution $\P{R>r}=\log(1+1/(r-1))$ at time $\alpha n$. So, one can couple both process again in the natural way.

The reason why we first work with the process $X^n$ instead of the process $W^n$ lies in a monotonous property presented in the construction. To make it clear, define the set $P_n=\{\ell /n\in [0,1): \ell \in \{0,1,...,n-1\}\}$ and associate it with the torus $\mathbb{Z}/n\mathbb{Z}=\{0,1,...,n-1\}=nP_n$.  Notice that in the process $X^n$, any element in the region $R_{\mathcal{O}(\ell,k)}$ covers the points $\{\ell,\cdots, \ell+k-1\}$ in $\mathbb{Z}/n\mathbb{Z}$, and also covers the points $\{\frac{\ell}{n},\frac{\ell +1}{n},...,\frac{\ell +k-1}{n}\}\in P_n$ in the Mandelbrot-Shepp model. Therefore, whenever the torus is covered in the process $X^n$, the points $P_n$ are also covered, and vice versa.  This fact is not true for the process $W^n$, where it is possible to cover the set $P_n$ in the Mandelbrot-Shepp model, but not cover the set $\mathbb{Z}/n\mathbb{Z}$ using the process $W^n$. 

In the coupling of $X^n$, observe that it is possible for the truncated configuration $\omega[1/n]$ to cover the set $P_n$ but not necessarily the whole interval $[0,1)$. To address this case, before proving the limit in distribution of the covering, we need to construct a quantitative argument by computing the number of points in the set $P_n$ that are missing in the process $X^n$, under the condition that the Mandelbrot-Shepp model is not covered; see Proposition \ref{teo:circlecontdimention}, where we show that whenever a point is missing in the set $\omega[1/n]$, with high probability there must also be a missing point in the set $P_n$.

Given any configuration $\omega=\sum_{i\in I} \delta_{(u_i,r_i)}$, define for any $M>0$ and $z>0$, the truncated  configuration above $M$ and bellow at height $z$ as:
\begin{align*}
  \omega[z,M]=\sum_{i\in I} \delta_{(u_i,r_i)}\delta_{\{Z>r_i>z\} }.
\end{align*}  With this, conditioning that the origin is not covered, we have a small region near the origin with many  missing points.



\begin{Lem}\label{teo:linecontdimention}
Fix $\zeta\in\left(0,1\right)$, and $\alpha\in (0,1)$. For any integer $n>0$, any value of $M>0$ such that $M+\zeta<1$, and an arbitrarily $r=r(n)\in\left[0,\frac{1}{n}\right)$, define:
\begin{align*}
    Y(n,\zeta,r,M)=\sum_{k=1}^{\lfloor \zeta n\rfloor} \ind\left\{\frac{k}{n}+r\in \mathcal{V}\left(\omega\left[\frac{1}{n},M\right]\right)\right\},
\end{align*} the number of vacant points of the form $\frac{k}{n}+r$ in the open set $(0,\zeta)$, with respect to the truncated measure $\omega\left[\frac{1}{n},M\right]$.  Then, for every positive $\e>0$, we get that:
 \begin{align} \label{eq:Yconcentrationholes}
\lim_{n\to \infty} \Pa{\left|\frac{\ln{Y(n,\zeta,r,M)}}{\ln{n}}- (1-\alpha)\right|>\e\middle| \,0\in \mathcal{V}\left(\omega\right) }=0. 
\end{align}
\end{Lem}
\begin{proof}[Proof of Lemma \ref{teo:linecontdimention}]
The proof is based on the concentration results of the branching process; see Theorem \ref{thm:Concentrationbranching}. Here, the main strategy is to divide the covering into height scales, and in each scale define two independent processes, a supercritical branching process and an ignition process. The idea consists of proving that many ignitions will occur with high probability, and for each such ignition we can obtain an independent branching processes with positive probability to survive. Whenever the branching process survives, we will be able to concentrate the random variable $Y(n,\zeta,r,M)$ around $n^{1-\alpha}$, the expected number of children of the branching process. 

To define the height scales of the problem, with fixed $\alpha>0$ take any $N=N(\alpha,\zeta)$ that satisfies $N>\max\{6,2\zeta^{-1}\}$ and $\frac{1}{2}N^{1-\alpha}\exp\left\{-\alpha\right\}>1$. Here, the first condition guarantees a minimal value to start the branching process described in equation~\eqref{eq:Branchingdef1} and proceed with it between scales, and the second condition is used to guarantee a high expected number of children in the construction of the branching processes, see equation~\eqref{eq:Expxiconstruction1}. 

With the value of $N$ fixed, divide the cylinders into height scales $H(\ell)=[0,1)\times (N^{-(\ell+1)},N^{-\ell}]$, where $\ell\geq 1$. Also, to simplify the notation, assume that $M=\zeta$ and $n=N^{-\ell_f}$ for some integer $\ell_f\geq 1$.  Later in the proof, we can show that this assumption does not affect the results. 

To proof the limit in equation \eqref{eq:Yconcentrationholes}, we are going to show that for any $\delta>0$, and for any $\e>0$, there exists a value of $n_0=n_0(\alpha,M,\e,\delta)$ such that for every $n>n_0$:
\begin{align}
\Pa{\left|\frac{\ln{Y(n,\zeta,r,M)}}{\ln{n}}- (1-\alpha)\right|>\e\middle| \,0\in \mathcal{V}\left(\omega\right) }<\delta. 
\end{align} For this, we will also divide the proof in two. The first part is to show that $Y(N^{\ell_f},\zeta,r,M)$ is greater than $N^{\ell_f(1-\alpha-\e)}$  with high probability. Then using first moment techniques, the second part consists in showing that $Y(N^{\ell_f},\zeta,r,M)$ is smaller than $N^{\ell_f(1-\alpha+\e)}$ with high probability.  

In each scale $(H(\ell))_{\ell}$, define the following set of intervals used in the construction of the branching process:
\begin{align*}
    \mathcal{I}(\ell)=\left\{\left[\frac{2k}{N^{\ell}}-r,\frac{2k+1}{N^{\ell}}-r\right):2\leq k\leq \frac{(\zeta +r)N^{\ell}-1}{2} \right\}.
\end{align*} Where $ \mathcal{I}(\ell)$ is the set of two by two disjoint intervals that does not exceeds the value of $\zeta$. Also, it does not contain the first two possible intervals (do not starts at zero). Furthermore, in the Mandelbrot-Shepp model, notice that each sequence of fitted vacant intervals $\{(I_\ell)_{\ell=1}^{\ell_f}: I_\ell\in \mathcal{I}(\ell)\}$ contributes to one element in $Y(n,\zeta,r,M)$; therefore, we can bound $Y(n,\zeta,r,M)$ bellow by the number of such sequences. See Figure \ref{fig:6} for a representation of the set $H(\ell)$ and the regions $\mathcal{I}(\ell)$.

\begin{figure}[ht!]
    \centering
    \includegraphics{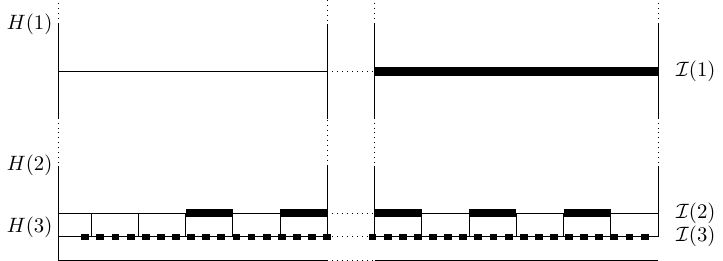}
    \caption{The space divided into regions $H(\ell)$, and in each region $H(\ell)$, we draw the set $\mathcal{I}(\ell)$ of disjoint segments, where we allow the branching process to survive. Notice that the first rectangle is always deform by the point $r$, and the set $\mathcal{I}$ does not start in the beginning. }
    \label{fig:6}
\end{figure}

To count the number of vacant intervals, fix a point $x\in [0,\zeta)$ and $\ell\geq 1$, and define the event:
\begin{align*}
     B^\ell(x)&=\left\{\left[x-\frac{1}{N^{\ell+1}},x\right)\in \mathcal{V}\left(\omega\left[N^{-(\ell+1)},N^{-\ell}\right]\right)\right\}
\end{align*}  For any $\ell\geq 1$, fixed any $I\in \mathcal{I}(\ell)$, denote by $I^0=\inf\{i: i\in I\}$ its first point, and define the random variable:
\begin{align*}
    \xi_I^\ell=\sum_{j=1}^{\lfloor N/2\rfloor} \ind\left\{B^{\ell} \left(I^0+\frac{2j}{N^{\ell+1}}\right)\right\}.
\end{align*}  In essence, the number of two by two disjoint empty regions from $I(\ell+1)$ that lies in $I$, which are not covered using the objects that are in $H(\ell)$.

As the construction suggests, we call $\xi_I^{\ell}$ as the number of children in the interval $I$. Notice that for any distinct $I,J\in \mathcal{I}(\ell)$, the random variables $\xi_J^{\ell}$ and $\xi_I^{\ell}$ are independent, since the distance between the intervals is greater than the larger object revealed by each event. Moreover,  with $\ell+1\leq \ell_f$  and for every point not near the origin and not near the end $\zeta$ (more precisely, $\frac{1}{N^{\ell}}\leq x-\frac{1}{N^{\ell+1}}<x<\zeta$), we have the following.
\begin{align*}
    \Pa{B^{\ell}(x)\middle |0\in \mathcal{V}\left(\omega\left[\frac{1}{n},M\right]\right)}&=\exp\left\{\int\limits_{N^{-(\ell+1)}}^{N^{-\ell}}\frac{\alpha}{y} \,dy\right\} \exp\left\{-\int\limits_{N^{-(\ell+1)}}^{N^{-\ell}}\frac{1}{N^{\ell+1}}\frac{\alpha}{y^2} \,dy\right\}\\&= N^{-\alpha} \exp\left\{-\alpha \left(1-\frac{1}{N}\right)\right\}.
\end{align*} Witch implies, that the expected number of children is greater than: 
\begin{align}\label{eq:Expxiconstruction1}
    \mathbb{E}_{\alpha}\left(\xi^{\ell}_I\right)\geq  \left\lfloor\frac{N}{2}-1\right\rfloor N^{-\alpha}\exp\left\{-\alpha \left(1-\frac{1}{N}\right)\right\}>N^{1-\alpha}\frac{\exp\left\{-\alpha\right\}}{2}>1 .
\end{align} 

Finally, given any $I\in  \mathcal{I}(\ell)$, define the branching process $(Z_{i}^I)_i$ as follows: Let $Z_0^I=1$ and $\Gamma_0=I$. Then, inductively, in generation $(i-1)-th$ given any set of intervals $\Gamma_i=\{J_k: J_k\in \mathcal{I}(\ell +i -1)\}_k$, define 
\begin{align}\label{eq:Branchingdef1}
    Z_{i}^I&=\sum_{J\in \Gamma_i} \xi^{\ell+i}_J, \text{ and} \\
    \Gamma_{i+1}&= \left\{L\in I(\ell+i): L\subset J\in \Gamma_i \text{ and } L\subset \mathcal{V}\left(\omega\left[N^{-(\ell+ j)},N^{-(\ell+j-1)}\right]\right)  \right\} \nonumber.
\end{align} 

 The mean number of children in each generation is greater than one by equation \eqref{eq:Expxiconstruction1}. So, we have a supercritical branching that can survive indefinitely with positive probability. Since we do not exactly know the distribution of the number of children in each generation, we need to perform an indirect calculation. By Theorem \ref{thm:Concentrationbranching}, for every positive $\e>0$ and any fixed height $\ell_0>1$, there exists a probability $\theta=\theta(\alpha,N)>0$ to survive at the limit when $n$ goes to infinity. By symmetries of the problem, this limit probability does not depend on the height of the first interval $I\in \mathcal{I}(\ell_0)$, 
since the distribution of the number of children of the process does not change between heights. In particular, for any $\e>0$, there exists $\ell_1=\ell_1(\theta,\e,\alpha,N)$ such that whenever $\ell_f>\ell_1+\ell_0$, we get for every $I\in \mathcal{I}(\ell_0)$ that:
\begin{align}\label{eq:ConcentrationgivenbyTHM2}
    \P{\frac{Z_{N^{\ell_f}}^I}{N^{(\ell_f-\ell_0)(1-\alpha)}}\geq N^{-(\ell_f-\ell_0)\e }}\geq \frac{\theta}{2}.
\end{align}  Later in the proof , we will use this equation and the value of $\ell_1$ to give a positive bound on the probability to survive in the last scale with many points in $Y(n,\zeta,r,M)$ using  independent trials that belongs to different initial heights. 

With the branching well defined, it is time to construct the ignition process as a sequence of events $(E_{\ell})_{\ell=1}^{\ell_2}$ for some $\ell_2\in\{1,\cdots, \ell_f\}$. Such sequence of events is not independent and will be determined by a set of objects near the origin. The goal of the process is to guarantee the existence of seeds, where each seed gives birth to many new independent branching processes capable of surviving.

For each $\ell\geq 1$, define the region $R(\ell)=[0,N^{-(\ell-1)}]\times(N^{-\ell},M)$. Then, set the event $E_{\ell}$ as: 
\begin{align}\label{eq:ignitiondefevent}
    E_{\ell}=\{\omega(R(\ell))=\emptyset\}.
\end{align}Despite being a sequence of dependent events, it satisfies the following property that guarantees many occurrences when we have a large number of scales to look.
\begin{Lem}\label{lem:intermediaryYlemma}
    For every $\delta>0$, $\alpha>0$, $M>0$, and $N>0$, define the ignition event as equation \eqref{eq:ignitiondefevent}. For every $J>0$, there exists a fixed integer $\ell_2=\ell_2(\delta,J,\alpha,N,M)<\infty$ such that whenever $\ell_f>\ell_2$, we get that: 
    \begin{align*}
        \Pa{\sum_{\ell=1}^{\ell_2} \ind\{E_\ell\}>J\middle| 0\in \mathcal{V}\left(\omega\left[\frac{1}{N^{\ell_f}},M\right]\right)}>1-\frac{\delta}{2}.  
    \end{align*}
\end{Lem}

The proof of Lemma \ref{lem:intermediaryYlemma} is postponed until the end of this proof. Moreover, to give an idea of the proof, we use the fact that the distribution of the closest object to the origin in each scale has a similar law and does not have mass at the origin.

Observe that, since $N>6$, whenever an ignition event $E_{\ell}$ occurs, we can say that the interval $\left[\frac{4}{N^{\ell+1}}-r, \frac{5}{N^{\ell+1}}-r\right)\in \mathcal{I}(\ell+1)$ is completely empty at height $\omega[N^{-(\ell+1)}]$. This allows a branching process to start there. Moreover, such an exploration of the branching process will be independent of the upcoming ignition events because of their mutual distance. 

Therefore, we can guarantee a large vacant set in the limit with high probability. With fixed $\e>0$, the idea consists of fixing some $J=J(\theta,\delta)$ large such that $\left(1-\frac{\theta}{2}\right)^J<\frac{\delta}{2(1-\delta)}$, then using the ignition Lemma \ref{lem:intermediaryYlemma}, we can find a minimal height $\ell_2$ such that with high probability there are $J$ ignitions until height $\ell_2$. In particular, if there are less than $J$ ignitions, we will assume that the vacant set is small, but if there are more than $J$, each of them will give an independent chance to survive with a large set in the branching, thus using the minimal height $\ell_1$ from equation \eqref{eq:ConcentrationgivenbyTHM2}, and taking $\ell_f>\ell_2+\ell_1$, we get that:
\begin{align*}
    \P{\frac{\ln{Y(N^{\ell_f},\zeta.r,M)}}{\ln{N^{\ell_f}}}\leq 1-\alpha-\e \middle |0\in \mathcal{V}\left(\omega\left[\frac{1}{N^{\ell_f}},M\right]\right)}&\leq \frac{\delta}{2}+\left(1-\frac{\delta}{2}\right)\left(1-\frac{\theta}{2}\right)^J<\delta.
\end{align*}

To prove that $Y$ cannot be greater than $N^{1-\alpha+\e}$, we can use the Markov inequality. Computing its first moment, there exists a $c=c(n,r)$ such that: 
\begin{align*}
    \mathbb{E}_{\alpha}\left({Y(n,\zeta,r,M)}\right)=\sum_{k=1}^{n}\P{\frac{k}{n}+r \in \mathcal{V}\left(\omega\left[\frac{1}{n}\right]\right)} \leq c n^{1-\alpha}.
\end{align*} So, for every $\e>0$ and $\delta>0$, there exits $\ell_3=\ell_3(\alpha,N,\e,\delta)$, such that for every $\ell_f>\ell_3$, we get that:
\begin{align}\label{eq:okmascomcondicional}
   \Pa{\left|\frac{\ln{Y(N^{\ell_f},\zeta.r,M)}}{\ln{N^{\ell_f}}}-(1-\alpha)\right|\leq \e \middle|0\in \mathcal{V}\left(\omega\left[\frac{1}{N^{\ell_f}},M\right]\right)}<\delta.
\end{align} 

Now the proof is almost over, but we need to do a few small considerations.  The first one is about the conditional event, and the second one is about the choice of $M$ and $n$ made early in the proof. 

We can exchange the conditional event in equation \eqref{eq:okmascomcondicional} for the event $\{0\in \mathcal{V}(\omega)\}$, for this, notice that $Y(n,\zeta,r,M)$ is a random variable that depends only on the configuration $\omega\left[\frac{1}{n},M\right]$. In particular, by independence of the Poisson random variable, such event is independent from objects smaller than $1/n$ or bigger than $M$, so for every $y_1\leq 1/n$, and $y_2\geq M$, we get that:
\begin{align*}
    &\P{\left|\frac{\ln{Y(n,\zeta.r,M)}}{\ln{n}}-(1-\alpha)\right|\leq \e \middle|0\in \mathcal{V}\left(\omega\left[\frac{1}{n},M\right]\right)}\\&=\P{\left|\frac{\ln{Y(n,\zeta.r,M)}}{\ln{n}}-(1-\alpha)\right|\leq \e \middle|0\in \mathcal{V}\left(\omega\left[y_1,y_2\right]\right)}.
\end{align*} Therefore, taking $y_1\to 0$, and $y_2\to \infty$, we can conclude that both sequence have the same limit. 

The last step to finish the proof is to show that the result can be obtained regardless of the choice of $n=N^{\ell_f}$ and $M$. To show that the same limit holds for any value of $n$, observe that on scales $N^{-\ell}$, by construction, we can guarantee completely vacant regions with probability $e^{-\alpha N}$, thus for values of $n$ between $N^{\ell+1}$ and $N^{\ell}$, we have by law of large numbers that $Y(n,\zeta,r,M)$ is greater than $e^{-\alpha N}$ times $Y(N^{-\ell},\zeta,r,M)$, giving the desired concentration. For the choice of $M$, notice that the limit occurs due to the existence of some ignition events defined by Lemma \ref{lem:intermediaryYlemma}; therefore, for any choice of $M$, we can always look for ignitions smaller than $M$ to guarantee the same survival rate. In this proof, we just choose $M+\zeta<1$, to avoid interference with the fact that $\{0\in \mathcal{V}\}$. With these considerations, we finish the proof. 
\end{proof}

For simplicity, the following proof inherits all the previous definitions. 

\begin{proof}[Proof of Lemma \ref{lem:intermediaryYlemma}]

The proof of this lemma consists of a dynamical construction that explores the Mandelbrot-Shepp set from top to bottom trying to find ignition events. Such construction induces a renewal process in the scales, and since such renewal is formed by random variables with well-behaved moments, the lemma is a direct consequence of the weak law of large numbers. 

We start the construction with fixed values of $\alpha>0$, $\delta>0$, $\zeta>0$, $M>0$ and $N>0$. In the construction, we will define the first region to have all the irregularities of the problem, so that the next ones are simpler and recursive. Here, we track three main information in each step: a region, a height, and the closest point to the $y-$axis within this region. Define a scale $s_0=\inf\{\ell\geq 1: N^{-(\ell-1)}<\min\{\zeta,M\}\}$, and set:
\begin{align*}
    h_0&=N^{-s_0}. \\
    A_0&=[0,\zeta]\times[h_0,M]. \\ 
    d_0&=\sup\{d: \omega(A_0\cap[0,d]\times(0,\infty))=0\}.  
\end{align*}

In words, we can describe our procedure as follows: Given a fixed region, look for the closest point to the  $y-$axis and find its distance $d$. If $d$ is large, it might be the case that you found an ignition. However, if $d$ is small, the object can influence the ignition event on other scales, so we must use $d$ to find the next scale not influenced by the objects discovered so far. Then, in the next undisturbed region, we repeat this procedure until we find many ignitions and prove the lemma.

With $A_0$ fixed, the random variable $d_0$ is an exponential random variable with some fixed positive rate, so it is not zero with probability one. In this proof, we are going to count the number of ignition events that occur just below $N^{-s_0}$, and show that by looking to deeper scales, we can find as many as we want. Since $\{d_0>0\}$ with probability one, we can define the triplet $(A_1,h_1,d_1)$ inductively. To be specific, with a fixed triplet $(A_{k-1},h_{k-1},d_{k-1})$ where $\{d_{k-1}>0\}$, we will define a height $s_{k}=\inf\{\ell\geq 1: N^{-(\ell-1)}<d_{k-1}\}$, then set:
\begin{align*}
    h_k&=N^{-s_k}.\\
    A_k&=[0,N^{-(s_k-1)}]\times [h_k,h_{k-1}).\\
    d_k&=\sup\{d: \omega(A_k\cap[0,d]\times(0,\infty))=0\}.
\end{align*}Moreover, we have that $d_k$ is an exponential random variable with rate $\alpha \left(h_k^{-1}-h_{k-1}^{-1}\right)$, so it is positive with probability one. 

By the continuity of the exponential random variable, it is clear that the process constructed above can be repeated infinitely many times. However, this is not sufficient to prove the lemma. To finish, we need to control the distance between different scales in the sequence and compute the probability of having an ignition in any fixed step. 

Define $Z_k=s_{k+1}-s_{k}$ the distance between the scales and then notice that for every positive integer $x>0$, the event $\{Z_k\geq x\}$ occurs if and only if $\{d_k<N^{-(s_k+x-1)}\}$ also happens. Thus, we have:  
\begin{align*}
    \P{Z_k\geq x}&=\P{d_k\leq N^{-(s_k+x-1)}}\\
    &=1-\exp\{-\alpha N^{-(s_k+x-1)}(h_k^{-1}-h_{k-1}^{-1})\}\\
    &\leq 1-\exp\{-\alpha N^{-x+1}\}.
\end{align*} In particular, the scale distance between different elements in the dynamical construction is a random variable with a light tail that has all the moments. 

As a direct consequence of this, for every $\delta>0$ and $J_0>0$, one can find $K_0(J_0,N,\delta)$ such that for every $n>K_0$: 
\begin{align*}
    \P{\sum_{i=1}^{n} Z_k >J_0}>1-\frac{\delta}{2}. 
\end{align*}

Now, in each of the $J_0$ steps, there might be an ignition happening. To compute the probability of an ignition on the scale $s_k$, note that if $\{N^{-(s_{k}-1)}<d_{k-1}\}$, then no object greater than $h_{k-1}$ can influence the presence of the ignition. Furthermore, in each step, we only need to check if $\{d_k>N h_k\}$ occurs. In particular, we have that:
\begin{align*}
    \P{d_k>N h_k}=\exp\{-\alpha N h_k \left(h_k^{-1}-h_{k-1}^{-1}\right) \}\geq e^{-\alpha N}. 
\end{align*} Where, for $N>0$ fixed is a constant probability bound. 

Therefore, after the scale $s_0=s_0(N,M,\zeta)$, since $J_0$ can be arbitrarily large, for every $J>0$ we can find a $K=K(\delta.J,\alpha,N,M,\zeta)$ such that for every $\ell>K$, we get: 
    \begin{align*}
        \P{\sum_{\ell=1}^{K} \ind\{E_\ell\}>J\middle| 0\in \mathcal{V}\left(\omega\left[\frac{1}{N^{\ell}},M\right]\right)}>1-\frac{\delta}{2}.  
    \end{align*}Finishing the proof as desired.

\end{proof}

\bigskip

\paragraph{The process $X^n$} We are going to adapt the result of the Lemma \ref{teo:linecontdimention} to the circle. In particular, exclusively for the process $X^n$, remember that if the truncated Mandelbrot-Shepp process covers the space, then $X^n$ also does. Now, we are going to prove that whenever the Mandelbrot-Shepp model is not covered, we can find many points in the form $\left\{\frac{k}{n}: k\in \{0,1,\cdots, n-1\}\right\}$ that are missing with high probability, thus forcing the process $X^n$ to be vacant as well. 

\begin{Prop}\label{teo:circlecontdimention} In the Mandelbrot-Shepp model, 
 let $n\in \N$, and set:
\begin{align*}
    Z_n=\sum_{k=1}^{n} \ind\left\{\frac{k}{n}\in \mathcal{V}\left(\omega\left[\frac{1}{n}\right]\right)\right\}
\end{align*} the number of missing points of the form $\left\{\frac{k}{n}: k\in \{0,1,\cdots, n-1\}\right\}$ in the circle $\mathbb{S}^1$, when we try to cover using just the truncated objects $\omega[1/n]$.  Then, we have that for $\alpha \in (0,1)$:
\begin{align}\label{eq:Zncomw}
      \frac{\log{Z_n}}{\log{n}}\ind\{\mathcal{V}(\omega)\neq \emptyset\}\overset{D}{\implies} (1-\alpha)\ind\{\mathcal{V}(\omega)\neq \emptyset\}.
\end{align}
\end{Prop}
\begin{proof}[Proof of Proposition \ref{teo:circlecontdimention}]This proof has two steps. The first is to show that $Z_n$ cannot be significantly greater than $n^{1-\alpha}$. The second step focuses on finding a rectangular region $\{L<x<F\}\times \{0<y<E\}$ with two properties; The first property is that the region itself has not yet been explored, and the second property is that the point $L$ is not covered. Therefore, by Lemma \ref{teo:linecontdimention}, we will have about $n^{1-\alpha}$ empty points in this unexplored region, pushing the number of missing points to $n^{1-\alpha}$.

Observe that $\pi(\alpha)=1$, then the equation \eqref{eq:Zncomw} is trivially satisfied. Therefore, without loss of generality, fix $\alpha\in(0,1]$ so that $\pi(\alpha)<1$ for the rest of the proof. 

In order to show that $Z_n$ cannot have a greater quantity of missing points,  using first moment, we get that:
\begin{align*}
    \Ex{Z_n}=n \P{0\in \mathcal{V}\left(\omega\left[\frac{1}{n}\right]\right)}=e^{-\alpha}n^{1-\alpha}.
\end{align*}Therefore, by Markov's inequality, for every $\e>0$:
\begin{align*}
      \lim_{n\to \infty} \P{\frac{\ln Z_n}{\ln{n}}>1-\alpha +\e, \mathcal{V}(\omega)\neq \emptyset}=0.
\end{align*} Finishing one side of the proof.  

For the other side, we are going to show that for fixed $\delta>0$, and for any $\e>0$, we have that: 
\begin{align}\label{eq:boundbyapp}
      \lim_{n\to \infty} \P{\frac{\ln Z_n}{\ln{n}}<1-\alpha -\e, \mathcal{V}(\omega)\neq \emptyset}\leq \delta.
\end{align} In particular, since $\delta$ is arbitrary, we get that: 
\begin{align*}
      \lim_{n\to \infty} \P{\left|\frac{\ln Z_n}{\ln{n}}-(1-\alpha)\right| >\e, \mathcal{V}(\omega)\neq \emptyset}=0.
\end{align*} Concluding the proof. 

To obtain the bound \eqref{eq:boundbyapp}, we will approximate the event $\{\mathcal{V}(\omega)=\emptyset\}$. Start by defining the first non covered point of the space as:
\begin{align*}
    L&=\inf\left\{x\in[0,1): x\in \mathcal{V}\left(\omega\right)\right\}.
\end{align*} Also, set $L=1$ if the space is completely covered. 

The definitions of the next two random variables: $E$ and $F$ are more complex. In words, $E(L)$ will look at a set of fitted regions in which a large object may appear. The position of the $x$ axis of this object with respect to the point $L$ is defined as $F(L,E(L))$. To illustrate, see in Figure \ref{fig:7} a construction of these random variables. Moreover, the main importance of the random variable is that, given a triplet $\{(L,E,F)=(x,s,t)\}$, it will be possible to find a rectangular region that starts in $x$, ends in $t$ and has height $s$, that has not yet been explored, and with two properties: the point $x$ is not covered, and the objects that appear in this region cover points not yet discovered. More precisely, define: 
\begin{align*}
    E(x)&=\sup\{s\in \left(0,1-x\right): \omega\left([x,1-(x+s))\times [s,\infty)\right)=0\},\\
    F(x,s)&= \sup\{t\in(x,1-(x+s)): \omega\left([x,t)\times [s,\infty)\right)=0\}.
\end{align*} Also set $E(x)=0$ when the regions $[x,1]\times (0,\infty)$ is empty.

\begin{figure}[h!]
    \centering
    \includegraphics{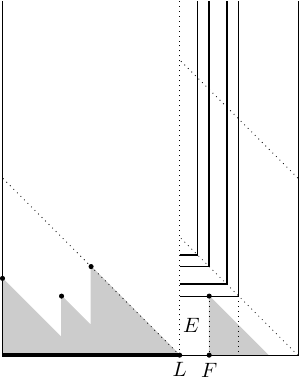}
    \caption{A representation of the random variables  $L$,$E$ and $F$. Given the point $L$, we can divide the space in three regions; one of them is empty, since $L$ is not covered, the other regions have the possibility of covering the interval $[0,L)$, the third region, is formed by objects not yet explored. The event $E$ looks to a region formed by the fitted rectangles, until they found a success draw as a point ( In the image the point is given by $(L+F,E)$). The distance of this point to $L$ is given by $F$. The gray regions, are the parts covered by object revealed in this construction}
    \label{fig:7}
\end{figure}

The main point of introducing the triplet $(L,E(L).F(L,E(L))$  is:
\begin{align}
    \Pa{\mathcal{V}(\omega)\neq \emptyset}&=\Pa{L<1,E(L)>0,F(L,E(L))>0} \nonumber \\
    &=\Pa{\bigcup_{N=1}^{\infty}\left\{L<1-\frac{1}{N},E(L)>\frac{1}{2N},F(L,E(L))>\frac{1}{2N}\right\}}.\label{eq:VELF}
\end{align} 

To prove equality \eqref{eq:VELF}, observe that the event $\{\mathcal{V}(\omega)\neq \emptyset\}$ is equal to the event $\{L<1,E(L ) \geq 0,F(L,E(L))\geq 0\}$. Since the events $\{E(L)=0\}$ and $\{F(L,E(L))=0\}$ have zero probability, we conclude the equality \eqref{VELF}. Here. we just ask for $\{E(L)>1/2N,F(L,E(L))>1/2N\}$ so when $L$ is close to $1-\frac{1}{N}$ the set of configurations that satisfy the event is not empty. 

To simplify the notation, define:
\begin{align*}
    G(N)=\left\{L<1-\frac{1}{N},E(L)>\frac{1}{2N},F(L,E(L))>\frac{1}{2N}\right\}.
\end{align*} Then, for every $\delta>0$, since we have a set of increasing events, we get that exists a $N_0(\delta)$ such that for every $N>N_0$, and for every $n\geq 0$, that: 
\begin{align*}
    \left|\P{\frac{\ln Z_n}{\ln{n}}<1-\alpha -\e, \mathcal{V}(\omega)\neq \emptyset}-\P{\frac{\ln Z_n}{\ln{n}}<1-\alpha -\e, G(N)}\right|\leq \frac{\delta}{2}.
\end{align*}

Now, with a fixed value of $N$, whenever $G(N)$ occurs, we will be able to give a lower bound to the quantity of $\frac{\ln Z_n}{\ln{n}}$. For this we are going to give a sequence of stochastic dominated random variables that will start in the random variable $Z_n$ given that the event $G(N)$ happens and ends in the random variable $Y$ from the Lemma \ref{teo:linecontdimention} given that the origin is empty. 

For this, given $\omega=\sum_{i\in I} \delta_{(x_i,r_i)}$ and $t>0$, define the shift:
\begin{align*}
    \phi_t(\omega)= \sum_{i\in I} \delta_{(t+x_i,r_i)},
\end{align*} where the sum is made in the circle. Also, for any $L\in[0,1)$, and $n>0$ define the quantity:
\begin{align*}
    r_n(L)=\inf\{r: n(L-r)\in\N\}.
\end{align*}

Now, assume that $G(N)$ occurs in some configuration $\omega$, then we will have a triplet $(L,E(L),F(L,E(L)))$, and with the triplet we can say that: 
\begin{enumerate}
    \item $L\in \mathcal{V}(\omega)$.
    \item $\omega((L,F)\times(E,1))=0.$
    \item The distribution of the objects in the region $\left(L,L+\frac{1}{2N}\right)\times \left(0,\frac{1}{2N}\right)$ is a Poisson random variable with rate $\Lambda_{\alpha}$, i.e. does not change. 
\end{enumerate}
Here, the facts $1.$ and $2.$ are a trivial consequence of the definition of the triplet $(L,E,F)$, and the claim $3.$ is true since the event $G(N)(\omega)$ given the random variable $L$, is completely determined by objects outside the region $\left(L,L+\frac{1}{2N}\right)\times \left(0,\frac{1}{2N}\right)$.

The first stochastic bound in $Z_n$ appears now, condition on the measure $G(N)$, since $\left\{F>\frac{1}{2N}\right\}$ occurs, by counting just the missing points in a smaller region, we have:
\begin{align*}
    Z_n \succeq \sum_{k=0}^{n-1} \ind \left\{\frac{k}{n}\in \mathcal{V}\left(\omega\left[\frac{1}{n}\right]\right)\right\} \ind \left\{L\leq \frac{k}{n}<L+\frac{1}{2N}\right\},
\end{align*} in the measure given the event $G(N)$. Here, $X\succeq Y$, if $X$ stochastically dominates $Y$.

Using the shift of the configuration $\omega$ to send $L$ to $0$, we can get that:
\begin{align*}
    Z_n(\omega) &\succeq \sum_{k=0}^{n-1} \ind \left\{\frac{k}{n}-r_n(L) \in \mathcal{V}\left(\phi_{-L}\left(\omega\left[\frac{1}{n}\right]\right)\right)\right\} \ind \left\{0\leq \frac{k}{n}-r(L)<\frac{1}{2N}\right\}.
\end{align*}

Now, we are close to relating $Z_n$ with the Lemma \ref{teo:linecontdimention}. The last thing to overcome is the restriction in the height $E>0$. For this, considering the condition $2.$, if we allow new independent objects to appear in $\left(L,L+\frac{1}{2N}\right)\times(E,1)$, we can cover more objects and thus diminish the value of $Z_n$ even more. Then, we have:
\begin{align*}
    Z_n &\succeq Y\left(n,\frac{1}{2N},r_n(L),1\right). 
\end{align*} where $Y$ count the number of missing points of the form $\left\{\frac{k}{n}: k\in \{0,1,\cdots n-1\}\right\}$ in the interval $\left(0,\frac{1}{2N}\right)$, and allow objects with length smaller than 1 to appear, objects that for the random variable $Z_n(\omega)$ given $G(N)$ is already known to not exists. Moreover, by Lemma \ref{teo:linecontdimention}, we concludes that exists a $n_0>0$ such that for every $n>n_0$ we get: 
\begin{align*}
    \P{\frac{\ln Z_n}{\ln{n}}<1-\alpha -\e, G(N)}\leq \frac{\delta}{2}.
\end{align*} In particular, for every $n>n_0$:
\begin{align*}
    \P{\frac{\ln Z_n}{\ln{n}}<1-\alpha -\e, \mathcal{V}(\omega)\neq \emptyset}<\delta. 
\end{align*}As desired.
\end{proof}

\begin{Cor}\label{Cor:Zncomwn}
Let, $n\in \N$, and set:
\begin{align*}
    Z_n=\sum_{k=1}^{n} \ind\left\{\frac{k}{n}\in \mathcal{V}\left(\omega\left[\frac{1}{n}\right]\right)\right\}.
\end{align*} the number of missing points of the form $k/n$ in the cylinder $S$, when we try to cover using just the truncated objects $\omega[1/n]$.  We have that, for $\alpha \in (0,1)$:
\begin{align} \label{eq:Zncomwn}
      \frac{\log{Z_n}}{\log{n}}\ind\{\mathcal{V}\left(\omega\left[\frac{1}{n}\right]\right)\neq \emptyset\}\overset{D}{\implies} (1-\alpha)\ind\{\mathcal{V}(\omega)\neq \emptyset\}.
\end{align}
\end{Cor}
\begin{proof}[Proof of Corollary \ref{Cor:Zncomwn}]
 Using that the random variable $\frac{\log{Z_n}}{\log{n}}$ is bounded, and $\{\mathcal{V}(\omega)\neq \emptyset\}\subseteq \{\mathcal{V}\left(\omega\left[\frac{1}{n}\right]\right)\neq \emptyset\}$. Furthermore, by Lemma \ref{lem:Topology}, we have $\ind\{\mathcal{V}\left(\omega\left[\frac{1}{n}\right]\right)\neq \emptyset\}\overset{D}{\implies}\ind\{\mathcal{V}(\omega)\neq \emptyset\}$. Then, the proof follows directly that equation \eqref{eq:Zncomwn} is the same as equation \eqref{eq:Zncomw}, plus some term that goes to zero in probability.
\end{proof}

With the Corollary \ref{Cor:Zncomwn}, we can proof Theorem \ref{teo:2*}, for the process $X^n$ with radius distribution $\log(1+1/(r-1))$ almost directly.

\begin{Prop}\label{Prop:Final}
    Consider $X^n(\omega)$ defined in equation \ref{eq:defXlog}, then we have that:
    \begin{align*}
        \lim_{n\to \infty} \Pa{X^n(\omega)=\mathbb{Z}/n\mathbb{Z}}=\pi(\alpha).
    \end{align*} Moreover, we have that:
    \begin{align}
         \frac{\log |\mathbb{Z}/n\mathbb{Z} \setminus X^n(\omega_{\alpha})|}{\log n}\ind\{X^{n}(\omega_{\alpha}) \neq \mathbb{Z}/n\mathbb{Z}_n\}&\overset{D}{\implies} (1-\alpha)\ind\{\mathcal{V}\neq \emptyset\}.
    \end{align}
\end{Prop}
\begin{proof}[Proof of Proposition \ref{Prop:Final}]
By the inclusion created by the coupling:
\begin{align*}
    \liminf_{n\to \infty} \P{X^n(\omega)= \mathbb{Z}/n\mathbb{Z}}\geq \liminf_{n\to \infty} \pi_{\frac{1}{n}}(\alpha)=\pi(\alpha),
\end{align*} Now, with the set $P_n=\left\{\frac{\ell}{n}\in [0,1): p\in\{0,1,2,...,n-1\right\}$, observe that:
\begin{align*}
    \Pa{X^n(\omega)= \mathbb{Z}/n\mathbb{Z}}&= \Pa{\mathcal{C}\left(\omega\left[\frac{1}{n}\right]\right)=[0,1)}+ \Pa{\mathbb{V}\left(\omega\left[\frac{1}{n}\right]\right)\neq \emptyset, P_n\subset \mathcal{C}\left(\omega\left[\frac{1}{n}\right]\right)} \\
    &=\pi_{\frac{1}{n}}(\alpha)+\P{Z_n=0, \mathcal{V}\left(\omega\left[\frac{1}{n}\right]\right)\neq \emptyset}.
\end{align*} In particular, by Lemma \ref{lem:continuityofprobabilitytruncated}, we can take the limit, and get that: 
\begin{align*}
    \lim_{n\to \infty} \P{X^n(\omega)= \mathbb{Z}/n\mathbb{Z}}=\pi(\alpha).
\end{align*}as desired. To finish the proof apply directly the Corollary \ref{eq:Zncomwn}, since $|\mathbb{Z}/n\mathbb{Z} \setminus X^n(\omega_{\alpha}))|$ is equal to the random variable $Z_n$. 
\end{proof}

The proof of Theorem \ref{teo:2*} is almost over, to complete the proof we need to connect the random process $X^n$ with the process $W^n$.

\paragraph{The process $W^n$} Unlike $X^n$, the process $W^n$ does not have a direct connection to the Mandelbrot-Shepp model. By construction, it is entirely plausible that the process $W^n$ is fully covered while the Mandelbrot-Shepp model is not, or, in another direction, it is possible that the Mandelbrot-Shepp model covers the set $P_n$ while $W^n$ is not completely covered. Despite the complications, coverage $X^n$ can be related to coverage $W^n$ in a simple way done in Proposition \ref{Prop:Final}, completing the proof of Theorem \ref{teo:2*} .

Observe that one side of the relation between $X^n$ with $W^n$ is trivial. Using Figure \ref{fig:5},  comparing the radius distribution in the coupling, one can conclude that the objects in $W^n$ have a radius of the same size or just one unit smaller than the radius of the objects in $X^n$. In particular, one may get that:
\begin{align}\label{eq:coveringXlowerY}
    \Pa{W^n(\omega)=\mathbb{Z}/n\mathbb{Z}}\leq \Pa{X^n(\omega)=\mathbb{Z}/n\mathbb{Z}}.
\end{align}

To finish the proof now, we need to prove that the covering using the law $X^n$, or analogously, the covering of the Mandelbrot-Shepp model, has the following high probability property: If the space is covered, then it is also covered by changing the sizes of all objects by one unit.

\begin{Prop}\label{Prop:Final2}
The sequence of probabilities:
\begin{align*}
   \left(\Pa{W^n(\omega)=\mathbb{Z}/n\mathbb{Z}}- \Pa{X^n(\omega)=\mathbb{Z}/n\mathbb{Z}}\right)_n 
\end{align*} is Cauchy. In particular: 
    \begin{align*}
    \lim_{n\to \infty} \Pa{W^n(\omega)=\mathbb{Z}/n\mathbb{Z}}=\pi(\alpha). 
    \end{align*}Moreover, we have that:
        \begin{align*}
         \frac{\log |\mathbb{Z}/n\mathbb{Z} \setminus W^n(\omega_{\alpha})|}{\log n}\ind\{W^{n}(\omega_{\alpha}) \neq \mathbb{Z}/n\mathbb{Z}_n\}&\overset{D}{\implies} (1-\alpha)\ind\{\mathcal{V}\neq \emptyset\}.
    \end{align*}
\end{Prop}
\begin{proof} 
The crucial part of the proof lies on noticing that: In both constructions, as $n$ grows, the set of regions $R$ and $\widehat{R}$ for the objects becomes thinner. And, compared to the process $X^n(\omega)$, the objects in $W^n(\omega)$ are equal in size or have one unit smaller.  

Consider an arbitrary configuration $\omega$ that covers the space. By Lemma \ref{lem:Topology}, there exists some value of $n$ for which $\omega[1/n]$ also covers the space; this implies that we can extract a covering with a finite number of objects. Now, considering a larger value of $m$ (where $m > n$), we observe that $X^m(\omega)$ continues to cover the space. However, this time it accomplishes this using only the objects present in $\omega[1/n]$. In particular, as we focus solely on these larger objects, the objects in $W^m(\omega)$ progressively approach and align with the objects in $\omega[1/n]$ as $m$ goes to infinity. As a consequence, we can reasonably expect that $W^m(\omega)$ covers the space as well for large values of $m$.

To get a reasonable limit on how large $m$ should be, we need to understand the properties of the configurations that cover the space at height $n$. For each configuration $\omega=\sum_{i\in I} \delta_{(x_i,y_i)}$, define the shift operation:
\begin{align*}
    \omega*h= \sum_{i\in I} \delta_{(x_i,hy_i)}.
\end{align*}Define also the random variable $G$ that measures how much one can shift down a configuration while still covering it, this is: 
\begin{align*}
    G=\sup\{\eta\in(0,1): \mathcal{C}(\omega*(1-\eta))=[0,1)\} 
\end{align*} and set $G=0$ if the space is not covered.

Informally, the random variable $G$ is a measure of how stable coverage is in the Mandelbrot-Shepp model. Notice that for each $\alpha>0$, we have that $\omega_{\alpha}*(1-\eta)\sim \omega_{\alpha(1-\eta)}$, thus shrink the space implies changing its space rate. Then, when the space is covered in some truncated level $\omega[1/n]$. Evaluating $\mathcal{C}(\omega[1/n])$, the space is covered almost surely by a finite union of open intervals, each interval having at least two intersections: one at the beginning of the interval and the other at the end. Analyzing the configuration $\mathcal{C}(\omega[1/n]*(1-\eta))$, each object shrinks and, in particular, the size of the intersections also diminish. The value of $G$ will correspond to how much we can shrink the objects' size and still see covering; or analogously, is an evaluation of how much one can change the rate and still see covering. 

Notice that, by Lemma \ref{lem:Topology}, the event $\{G=0\}$ happens, if and only if, we do not cover the space. Thus:
\begin{align*}
    \P{G>0|\mathcal{V}=\emptyset}=1.
\end{align*}  Therefore, since the events are increasing, we get that $\P{G>\eta}$ converges to $\P{\mathcal{V}=\emptyset}$ when $\eta$ goes to zero.  Thus, for fixed value of $\alpha>0$, and every $\e>0$, exists $\eta=\eta(\alpha)$ such that covering and having a positive value of $G$ is close in probability, this is: 
 \begin{align}\label{eq:piGsmall}
     |\pi(\alpha)-\Pa{G>2\eta}|<\e/2.
 \end{align}

In equation \eqref{eq:piGsmall}, with $P_n=\{k/n: k\in \{0,1,\cdots,n-1\}\}$, the probability of covering $P_n$ is $\P{X^n(\omega_{\alpha}=\mathbb{Z}/n\mathbb{Z}}$ and this probability converges to $\pi(\alpha)$ by Proposition \ref{Prop:Final}. To approach $\Pa{G>2\eta}$, notice that if for every $n$, we have that when $\{P_n\subset  \mathcal{C}\left(\omega_{\alpha}*(1-\eta)\left[\frac{1}{n}\right]\right)\}$ does not happen, then $\{\mathcal{V}\left(\omega_{\alpha}*(1-\eta)[1/n]\right)\neq \emptyset\}$, and by Lemma \ref{lem:Topology}, we get $\{\mathcal{V}\left(\omega_{\alpha}*(1-\eta)\right)\neq \emptyset\}$ happens. Moreover, since $\{G\geq 2\eta\}$ happens if and only if $\omega_{\alpha}*(1-\eta')$ is covered for every $\eta'<\eta$. Then, we get that, asking for a diminish of $\eta$ ( instead of $2\eta$, to avoid a topology problem in the boundary of the objects), we get:  
\begin{align*}
    \lim_{n\to \infty} \P{P_n\subseteq \mathcal{C}\left(\omega_{\alpha}*(1-\eta)\left[\frac{1}{n}\right]\right)\middle |G>2\eta}=1.
\end{align*}

Applying the limits in equation \eqref{eq:piGsmall}, we have:
\begin{align*}
  \limsup_n \left|\P{X^n(\omega_{\alpha})=\mathbb{Z}/n\mathbb{Z}}-\P{P_n\subseteq \mathcal{C}\left(\omega_{\alpha}*(1-\eta)\left[\frac{1}{n}\right]\right),G>2\eta}\right|<\e. 
\end{align*} In particular, whenever $X^n(\omega_{\alpha})$ covers the space and $n$ is large, we can diminish the size of the objects, i.e.  $X_{\alpha n}(\omega*(1-\eta))=\mathcal{C}\left(\omega_{\alpha}*(1-\eta)\left[\frac{1}{n}\right]\right)$, and still cover with high probability. Since this property starts at some value of $n$, and will be also true for bigger values. Then, comparing with the covering $W^n$, and using equation \eqref{eq:coveringXlowerY} we get that, when $n$ is large enough, both models cover the space with high probability. More precisely, we have that for $n$ sufficient large:
\begin{align*}
   \Pa{X^n(\omega)=\mathbb{Z}/n\mathbb{Z}}&\geq \Pa{W^n(\omega)=\mathbb{Z}/n\mathbb{Z}}\geq \P{X^n(\omega_{\alpha}*(1-\eta))=\mathbb{Z}/n\mathbb{Z}}\\
    &\geq \P{P_n\subseteq \mathcal{C}\left(\omega_{\alpha}*(1-\eta)\left[\frac{1}{n}\right]\right),G>2\eta},
\end{align*} In particular, we get for every $\e>0$ that:
\begin{align*}
  \limsup_n \left|\Pa{W^n(\omega)=\mathbb{Z}/n\mathbb{Z}}- \Pa{X^n(\omega)=\mathbb{Z}/n\mathbb{Z}}\right|<\e. 
\end{align*} Concluding that the sequence is Cauchy. 

To finish the proposition, we need to show the limit of distribution of the number of missing points in the covering $W^n$, for this just observe that whenever $W^n$ does not cover the torus in the limit when $n\to \infty$, neither does $X^n$. In particular, by construction, since the vacant set of $W^n$ is bounded by a constant times the size of the vacant set of $X^n$ ( the object has the same size, or one unit less). We get that both satisfy the same limit in distribution, since the constant does not change the limit. 
\end{proof}
\section{Open questions} \label{subsec:Open} \noindent

During this paper, we encountered many questions that were left open. To state them, consider the cover process $X_t$, with radius $R$ distributed as $f(r)$, then:  
\begin{enumerate}
    \item Describe the exact number of different phases of the covering process in the one dimensional case. To be more specific, complete the Gumbel Phase, showing that if $R\in L^1$, there exists a constant $c$ such that $\frac{T_n}{n\ln{n}}$ converges in probability to $c$, and find the limits in distribution for all function on the form $f(r)=\frac{1}{r\ln^b(r)}$, with $b>1$; In between the Gumbel phase and the compact phase it is expected a new phase, for this take $f(r)=\frac{1}{r\ln{r}\ln{\ln{r}}}$ and find its limit in distribution (notice that it is not expected to be like Gumbel since it is not have first moment, and $1/f(n)>n\ln{n}$, so it is not expected to be in the same scale as the compact phase; In the compact phase, find the exact conditions or proof that conditions \eqref{eq:Star} and \eqref{eq:Triangle} fully describe this phase. 
    \item Find the explicit function of $\pi$.
    \item The covering problem makes sense in high dimensions, and all the questions solved by this paper are partially unknown for dimension bigger then two. 
    \item Changing the shape of the object give rise to interesting questions. For instance, if consider objects with shape of random walks, one can recover the covering time for  random interlacements.  
\end{enumerate}  
\section{Appendex}
\label{sec:appendix}
\subsection{Useful propositions}
\label{subsec:Usefulprop}

Here, we are going to fill out some details that is left from the proof, or just enunciate them and indicate where the proof is. 

\begin{proof}[Proof of Lemma \ref{Lem:EqvH}]
To proof that $1$ is equivalent to $2$, observe that $R$ is discrete. Then, if $R\in \mathrm{L}^p(\R)$ for some $p>1$:
\begin{align*}
    \Ex{R^p}&=\sum_{y=1}^{\infty} y^p \P{R=y}= \int_0^{\infty} p \frac{y^pf(y)}{y} dy.
\end{align*}Therefore, if $y^pf(y)\to 0$, then for every $p'\in[1,p)$, it is true that $\Ex{R^{p'}}<\infty$. And if, $\Ex{R^p}<\infty$ for some $p>1$, then by Markov $y^pf(y)<\Ex{R^p}$, and for every $p''\in[1,p)$ it is true that $y^{p''}f(y)\to 0.$

To prove that $2$ equivalent to $1$ just check that for every $\lambda>0$, exists a $k_0=k_0(\lambda)$ such that for every $k>k_0$:
\begin{align*}
    f(k)k^{1+\frac{\lambda}{2}}<f(k)k^{1+\frac{\lambda}{2}}\ln{k}< f(k)k^{1+\lambda}<f(k)k^{1+\lambda}\ln{k}.
\end{align*} Finish the proof of the Lemma.
\end{proof}

\begin{proof}[Proof of Lemma \ref{Lem:Gumbel-3}]
    By straightforward calculations, one may get: 
    \begin{align*}
        \P{T_K^{\ell}< t}&=\P{\max_{k=1,...,K} \{\xi_k\}<t}=\prod_{k=1}^K \P{\xi_k<t}=(1-e^{-\frac{tp}{K}})^K. 
    \end{align*} Therefore:
    \begin{align*}
        \P{\frac{p}{K}T_K^{\ell}-\ln{K}< t}&= \left(1-\exp\left\{\frac{- Kp(t+\log{K})}{Kp}\right\}\right)^K= \left(1-\frac{e^{t}}{K}\right)^K
    \end{align*} That equation converges to the Gumbel distribution when $K$ goes to infinity, and does not depend on the parameter $p$. 
\end{proof}

\begin{proof}[Proof of the Lemma \ref{lem:poissoncontrol}]
For a Poisson Random variable $X$ with rate $\lambda$, using concentration, \cite{CCPPP}, it is true that for every $x>0$: 
\begin{align}\label{Chernoff}
    \P{|X-\lambda|>x}<2\exp\left\{-\frac{x^2}{2(\lambda+x)}\right\},
\end{align}

Now, with $\alpha>0$ notice that $\omega(R_n^1)$ is a Poisson random variable with rate $\alpha n$, and $\omega(R_n^{\delta})$ is a Poisson random variable with rate $ \alpha n/\delta$. So, we have:
\begin{align*}
    \mathbb{P}_{1+\e}\left({|\omega(R_n^1)-n(1+\e)|>\frac{\e}{2}n}\right) &\leq 2\exp\left\{-n\frac{\e^2}{2(2+3\e)}\right\}, \text{ and }\\
     \mathbb{P}_{1-\e}\left({\left|\omega(R_n^{\delta})-\frac{n}{1+\e}\right|>\frac{\e}{2(1+\e)}n}\right) & \leq 2\exp\left\{-n\frac{\e^2}{4(1+\e)(2(1-\e^2)+\e)}\right\}
\end{align*} Therefore
\begin{align*}
    \mathbb{P}_{1+\e}\left(\omega(R_n^1)<n\right)&\leq 2\exp\left\{-n\frac{\e^2}{4(2+3\e)}\right\},\text{ and}\\ 
    \mathbb{P}_{1-\e}\left(\omega(R_n^\delta)>n\right)&\leq 2\exp\left\{-n\frac{\e^2}{4(1+\e)(2(1-\e^2)+\e)}\right\}.
\end{align*} Since, both of them are summable, we conclude the Lemma.
\end{proof}

Some basic propositions about Regular Variation functions is expose here, the proof of them can be found in \cite{SIR}.

\begin{Prop}\label{prop:RVresults}
Then the following hold:
\begin{enumerate}
    \item If $L \in RV_0$, we have that $\lim_{x \to \infty} L(tx)/L(x)=1$ uniformly on each compact $t-$set of $(0,\infty)$.
    \item For every $U(x)\in \mathrm{RV}_p$, we have that $U(x)x^{-p}$ is slowly varying. Therefore, $U(x)=x^pL(x)$, for some $L\in \mathrm{RV}_0$.
    \item If $L \in RV_0$, then for every $\alpha>0$, $x^{\alpha}L(x)\to \infty$, and $x^{-\alpha}L(x)\to 0$.
    \item If $L \in RV_0$, then for every $\alpha> 0$, $(L(x))^{\alpha}$ and  $(L(x))^{-\alpha}$ are slowly varying.
    \item Let $L, L' \in RV_0$, then $L +L'$ and $LL'$ are slowly varying.
\end{enumerate}
\end{Prop}

Also proof in \cite{SIR}, Karamata's theorem, and Karamata's Representation Theorem

\begin{Theo}[Karamata's Theorem]\label{teo:Karamatastheorem} Considering the space of regular varying functions, we have that: 

(a) If $p\geq -1$ then $U\in \mathrm{RV}_p$ implies $\int_0^x U(t)dt\in \mathrm{RV}_{p+1}$ and
\begin{align} \label{eq:Karamatasequation}
    \lim_{x\to \infty} \frac{xU(x)}{\int_0^x U(t)dt}=p+1.
\end{align} If $p<-1$ (or if $p=-1$ and $\int_x^{\infty} U(s)ds<\infty$) then $U\in \mathrm{RV}_p$ implies $\int_x^\infty U(t)dt$ is finite, $\int_x^{\infty} U(t)dt\in \mathrm{RV}_{p+1}$ and
\begin{align}
    \lim_{x\to \infty} \frac{xU(x)}{\int_x^\infty U(t)dt}=-p-1.
\end{align}

(b) If $U$ satisfies
\begin{align}
    \lim_{x\to \infty} \frac{xU(x)}{\int_0^x U(t)dt}=\lambda\in (0,\infty)
\end{align} then $U\in \mathrm{RV}_{\lambda-1}$. If $\int_x^\infty U(t)dt< \infty$ and
\begin{align}
    \lim_{x\to \infty} \frac{xU(x)}{\int_x^\infty U(t)dt}=\lambda\in (0,\infty)
\end{align} then $U\in \mathrm{RV}_{-\lambda-1}$.
\end{Theo}

\begin{proof}[Proof of Lemma \ref{Lem:nontrivialcovering2}]
    The proof uses the same branching construction presented in Subsection ~\ref{subsec:nontrivialdistribution}. Since the proof is analogous to what was done there, some details will not be filled in.  
    
    Fix the tree $\widehat{\mathcal{T}}\subset \mathcal{T}_4$ defined in Subsection \ref{subsec:nontrivialdistribution}. Now, indexed in the vertices $v(i,h)\in \widehat{\mathbb{V}}$, define the regions in $[0,1)\times (0,\infty)$ and intervals in $[0,1)$ as:
    \begin{align*}
        \widetilde{R}(i,h)&=\left[\frac{|i|}{4^h},\frac{|i|+2}{4^{h}}\right)\times \left[\frac{1}{4^{h+1}},\frac{1}{4^h}\right)\\
        \widetilde{I}(i,h)&=\left[\frac{|i|+1}{4^h},\frac{|i|+2}{4^{h}}\right).
    \end{align*} Also, for $v(0,0)$, define $\widetilde{R}(0,0)=[0,1)\times [1/4,\infty)$, and $\widetilde{I}(0,0)=[0,1)$. See Figure \ref{fig:4}, where a representative of this regions in another scale happens. 

To define the Branching process, fix a configuration $\omega_{\alpha}$, and define $(Z_h)_h$, where $Z_0=1$, and associated with it, we have the vertex $v(0,0)$. For other values of $h$, define inductively $Z_{h+1}= \sum_{i=1}^{Z_{h}} 2\cdot \ind\{\omega(\widetilde{R}(|v(i)|,h)=0\}$,
where $\{v(1),...,v(Z_h)\}$ are the vertex associated with the $h-th$ generation of the branching process. Moreover, define the vertex associated with the next generation as the union of the two children in $\widehat{\mathcal{T}_4}$ of each $v\in \{v(i)\}_{i=1}^{Z_h}$ such that $\{\omega(\widetilde{R}(|v|,h)=0\}$ happens.  

As before, using the intervals $\widetilde{I}$ it is simple to show that if $\{Z_h>0\}$ then the event $\{\mathcal{V}(\omega[4^{-h}])\neq~ \emptyset\}$ happens. About probabilities, notice that the process is more homogeneous now, then:
\begin{align*}
    \Pa{\omega(\widetilde{R}(i,h))=0}&=e^{-6\alpha} \text{ for }h>0,\\
    \Pa{\omega(\widetilde{R}(0,0))=0}&=e^{-4\alpha}.
\end{align*} So despite the origin, by classical branching arguments for $\alpha<\frac{\ln(2)}{6}$ the event $\{Z_h>0, \forall h\}$ has positive probability. So clearly, by Lemma \ref{lem:Topology}:
\begin{align*}
    \Pa{\mathcal{V}(\omega)\neq\emptyset}=\Pa{\bigcap_h \left\{\mathcal{V}(\omega[4^{-h}])\neq\emptyset\right\}}>\P{Z_h>0, \forall h}>0.
\end{align*} As desired. 
\end{proof}

\subsection{Proving theorems in the discrete time case}
\label{subsec:Equivalence}

This subsection shows that theorems \ref{teo:1}, \ref{teo:2}, \ref{teo:3}, and \ref{teo:4}, are also true for the discrete covering model. For this, we need to show that the cover time of the problem in the continuous and in the discrete case have the same limit in distribution.

The equivalence of statements is due to the Proposition \ref{Prop:ExVar}, taking the pair $(a_n,b_n)$ as $(f^{-1}(n),0)$ or $(n,\ln{n})$.

\begin{Prop} \label{Prop:ExVar}
Let $a_n,b_n>0$, two sequence of numbers, with $a_n\to \infty$. Then, we have that:
\begin{align*}
   \frac{T_n}{a_n}-b_n\overset{D}{\implies}Y, \text{ if, and only if, }
   \frac{\tau_n}{a_n}-b_n\overset{D}{\implies}Y.
\end{align*}
\end{Prop}

\begin{proof}[Proof of Proposition \ref{Prop:ExVar}]
The proof follows using concentration inequalities for the difference of $\tau_n$ and $T_n$. To start, by the Poisson construction, set
\begin{align*}
    T_n&=\sum_{i=1}^{\tau_n} \eta_i,
\end{align*} where $\eta_i$ is a sequence of i.i.d. exponential random variables with rate $1$. By Chebyshev's inequality, we have
\begin{align}\label{eq:concentrationTt}
    \P{|k-\sum_{i=1}^{k} \eta_i| > k^{3/4}} = \P{|\sum_{i=1}^{k} (\eta_i-1)| > k^{3/4}}\leq k^{-1/2}.
\end{align} Since the cover times $\tau_n$ and $T_n$ are diverging in $n$, we have that condition on the value of $\tau_n$,  $T_n$ will not oscillate $(\tau_n)^{3/4}$ from this value. More than this, by symmetry we also  have that given $T_n$, the value of $\tau_n$ will not oscillate more than $T_n^{3/4}$ of this value.

To control both of this inequalities, we will make standard arguments using any open set $A\subset \R$, define $A_n$ such that:
\begin{align*}
    A_n=\{x\in A: a_n(x+b_n)\left[1+a_n^{-1/4} (x+b_n)^{-1/4}\right]\in A\}.
\end{align*}  We have that $A_n\uparrow A$, and:
\begin{align*}
    \P{T_n\in a_n (A+b_n)}\geq \P{T_n\in a_n (A+b_n),\tau_n\in a_n(A_n+b_n)}.
\end{align*} Then:
\begin{align*}
    \P{T_n\in a_n (A+b_n)}\geq \P{|T_n-\tau_n|< (\tau_n)^{3/4},\tau_n\in a_n(A_n+b_n)}.
\end{align*} Taking the limit, we can use the independence of the exponential variables, together with equation \eqref{eq:concentrationTt}, to show that:
\begin{align*}
    \liminf_{n\to \infty} \P{T_n\in a_n (A+b_n)}\geq \P{Y\in A}.
\end{align*} The other affirmation, have an analogous proof, just substitute $T_n$ by $\tau_n$.
\end{proof}


\bibliographystyle{plain}
\bibliography{ref}

\begin{thebibliography}{10}

\bibitem{ADiaconis}
David~J. Aldous.
\newblock Covering a compact space by fixed-radius or growing random balls.
\newblock {\em Latin American Journal of Probability and Mathematical Statistics}, 2022.

\bibitem{BP}
Leonard~E. Baum and Patrick Billingsley.
\newblock {Asymptotic Distributions for the Coupon Collector's Problem}.
\newblock {\em The Annals of Mathematical Statistics}, 36(6):1835 -- 1839, 1965.

\bibitem{JB}
Jean Bertoin.
\newblock {\em Subordinators: Examples and Applications}.
\newblock Springer Berlin Heidelberg, Berlin, Heidelberg, 1999.

\bibitem{BT}
Carina Betken and Christoph Thäle.
\newblock Approaching the coupon collector's problem with group drawings via stein's method, 2022.

\bibitem{CCPPP}
Clément Canonne.
\newblock A short note on poisson tail bounds, 2019.
\newblock Last accessed: 16/05/2022.

\bibitem{DB}
Brian Dawkins.
\newblock Siobhan's problem: The coupon collector revisited.
\newblock {\em The American Statistician}, 45(1):76--82, 1991.

\bibitem{Dv}
Aryeh Dvoretzky.
\newblock On covering a circle by randomly placed arcs.
\newblock {\em Proceedings of the National Academy of Sciences of the United States of America}, 42:199–203, 1956.

\bibitem{ER}
P.~Erdős and A.~Rényi.
\newblock On a classical problem of probability theory.
\newblock {\em The Annals of Mathematical Statistics}, pages 215--220, 1961.

\bibitem{FLM}
Victor Falgas-Ravry, Joel Danielsson, and Klas Markström.
\newblock Speed and concentration of the covering time for structured coupon collectors.
\newblock {\em Advances in Applied Probability}, 52, 01 2016.

\bibitem{LH}
Lars Holst.
\newblock On matrix occupancy, committee, and capture-recapture problems.
\newblock {\em Scandinavian Journal of Statistics}, 7(3):139--146, 1980.

\bibitem{MC-Diarmid}
Daniel Hsu.
\newblock Mcdiarmid’s inequality, 2020.
\newblock Last accessed: 17/06/2022.

\bibitem{SJ}
Svante Janson.
\newblock {Random coverings in several dimensions}.
\newblock {\em Acta Mathematica}, 156(none):83 -- 118, 1986.

\bibitem{KA}
Kahane Jean-Pierre.
\newblock {\em Some Random Series of Functions}.
\newblock Cambridge: Cambridge University Press, 2 edition, 1985.

\bibitem{AN}
Peter E.~Ney. Krishna B.~Athreya.
\newblock {\em Branching Processes}.
\newblock Springer Berlin, Heidelberg, 1 edition, 1972.

\bibitem{BBM}
Benoit~B. Mandelbrot.
\newblock Renewal sets and random cutouts.
\newblock {\em Z. Wahrscheinlichkeitstheorie und Verw. Gebiete}, 22:145--157, 1972.

\bibitem{MS}
Nathan Mantel and Bernard~S. Pasternack.
\newblock A class of occupancy problems.
\newblock {\em The American Statistician}, 22(2):23--24, 1968.

\bibitem{MFBBBRWK}
Joslin Moore, Michael Folkmann, Andrew Balmford, Thomas Brooks, Neil Burgess, Carsten Rahbek, Paul Williams, and Jakob Krarup.
\newblock Heuristic and optimal solutions for set‐covering problems in conservation biology.
\newblock {\em Ecography}, 26:595 -- 601, 10 2003.

\bibitem{MDP}
Mathew~D. Penrose.
\newblock {Leaves on the line and in the plane}.
\newblock {\em Electronic Journal of Probability}, 25(none):1 -- 40, 2020.

\bibitem{SIR}
Sidney~I. Resnick.
\newblock {\em Extreme Values, Regular Variation, and Point Process}.
\newblock Springer New York, NY, 1 edition, 2006.

\bibitem{Sh}
L.~A. Shepp.
\newblock Covering the circle with random arcs.
\newblock {\em Israel Journal of Mathematics}, 11:328–345, 9 1972.

\bibitem{LAS}
L.~A. Shepp.
\newblock Covering the line with random intervals.
\newblock {\em Z. Wahrscheinlichkeitstheorie und Verw. Gebiete}, 23:163--170, 9 1972.

\bibitem{XT}
Weiyu Xu and A.~Kevin Tang.
\newblock A generalized coupon collector problem.
\newblock {\em Journal of Applied Probability}, 48(4):1081--1094, 2011.

\end{thebibliography}

\end{document}